\documentclass{siamart171218}
\usepackage{amsfonts,amsmath,amssymb,graphicx,caption}

\makeatletter
\newcommand*\rel@kern[1]{\kern#1\dimexpr\macc@kerna}
\newcommand*\widebar[1]{%
  \begingroup
  \def\mathaccent##1##2{%
    \rel@kern{0.8}%
    \overline{\rel@kern{-0.8}\macc@nucleus\rel@kern{0.2}}%
    \rel@kern{-0.2}%
  }%
  \macc@depth\@ne
  \let\math@bgroup\@empty \let\math@egroup\macc@set@skewchar
  \mathsurround\z@ \frozen@everymath{\mathgroup\macc@group\relax}%
  \macc@set@skewchar\relax
  \let\mathaccentV\macc@nested@a
  \macc@nested@a\relax111{#1}%
  \endgroup
}
\makeatother
\usepackage{color}
\usepackage{MnSymbol}
\usepackage[ruled,algo2e]{algorithm2e}
\usepackage{scalerel,stackengine}
\stackMath
\newcommand\reallywidehat[1]{%
\savestack{\tmpbox}{\stretchto{%
  \scaleto{%
    \scalerel*[\widthof{\ensuremath{#1}}]{\kern-.6pt\bigwedge\kern-.6pt}%
    {\rule[-\textheight/2]{1ex}{\textheight}}%WIDTH-LIMITED BIG WEDGE
  }{\textheight}% 
}{0.5ex}}%
\stackon[1pt]{#1}{\tmpbox}%
}

\newcommand{\RR}{{\mathbb{R}}}

\newtheorem{Cor}[theorem]{Corollary}
\newtheorem{example}[theorem]{Example}
\newtheorem{remark}[theorem]{Remark}
\newtheorem{num_example}[theorem]{Example}

\title{Numerical methods for large-scale Lyapunov equations with symmetric banded data
\thanks{Version of April 13, 2018. Part of this work was supported by the Indam-GNCS 2017 Project
``Metodi numerici avanzati per equazioni e funzioni di matrici con struttura''. }}
\author{Davide Palitta\thanks{Dipartimento di Matematica, Universit\`a di Bologna, %
Piazza di Porta S. Donato, 5, I-40127 Bologna, Italy ({\tt davide.palitta3@unibo.it}).}  
\and Valeria Simoncini\thanks{Dipartimento di Matematica, Universit\`a di Bologna, 
Piazza di Porta S. Donato, 5, I-40127 Bologna, Italy ({\tt valeria.simoncini@unibo.it}),
and IMATI-CNR, Pavia, Italy.}}

\begin{document}
\bibliographystyle{siam}
\maketitle

\begin{abstract}
The numerical solution of large-scale Lyapunov matrix equations 
with symmetric banded data has so far received little attention in the rich literature
on Lyapunov equations. We aim to contribute to this open problem
by introducing two efficient solution methods, which respectively address the
cases of well conditioned and
ill conditioned coefficient matrices. The proposed approaches
conveniently exploit the possibly hidden structure of the solution matrix so
as to deliver memory and computation saving approximate solutions. 
%{\color{red} that is a low memory allocation and a small flops count are requested for their computation}.
Numerical experiments are reported to illustrate the potential
of the described methods.
\end{abstract}



%%%%%%%%%%%%%%%%%%%%%%%%%%%%%%%%%%%%%%%%%%%%%%%%%%%%%%%%%%%%%%%%%%%%%%%%%%%%%%%%%%%%%%%%%%%%%%%%%%%%%%%%%%%%%%%%%%%%%%%

\section{Introduction}
We are interested in the numerical solution of the large-scale Lyapunov equation
\begin{equation}\label{main.eq}
AX+XA=D, 
\end{equation}
where $A\in\RR^{n\times n}$ is symmetric and positive definite (SPD), 
$D\in\mathbb{R}^{n\times n}$ is symmetric, and both are large and banded matrices with 
bandwidth $\beta_A,\beta_D$, respectively. These hypotheses will be assumed throughout
the manuscript.
Lyapunov matrix equations play an important role in signal processing and control theory, 
see, e.g., \cite{Antoulas.05},\cite{Gajic1995},\cite{Datta1994}. However, 
they also arise in different contexts such as in
the discretization of certain elliptic partial differential 
equations, see, e.g., \cite{Palitta2015},
or as intermediate steps in nonlinear equation solvers,
 like the algebraic Riccati equation \cite{Bini2012}.

With the given hypotheses the solution matrix $X$ to (\ref{main.eq}) is symmetric,
and  it is positive (semi-)definite if $D$ is positive (semi-)definite.
In general, the matrix $X$ is dense and for large scale problems it
cannot be explicitly stored. A special situation arises when
$D$ is low rank, that is
%To overcome this numerical difficulty,
%an extensive literature has treated equation \eqref{main.eq} with 
$D=BB^T$, $B\in\mathbb{R}^{n\times s}$, $s\ll n$. In this case,
and under certain assumptions on the spectrum of $A$,
%if $A$ is not too ill-conditioned, 
$X$ can be well approximated by
a low-rank matrix, that is $X\approx ZZ^T$ with
% Indeed, having a low-rank right-hand side 
%is a necessary condition for showing a decay in the singular values of the solution $X$.
% See, e.g., \cite{Simoncini2016} and references therein.
% This means that $X$ can be well-approximated by a low-rank matrix 
$Z\in\mathbb{R}^{n\times t}$, $t\ll n$, so that only the tall matrix $Z$ needs
to be stored. %, $X\approx ZZ^T$. 
A rich literature is available for this setting, and successful ``low-rank'' algorithms
for large dimensions have been developed. 
%The so-called low-rank methods try to directly compute and store 
%only the tall matrix $Z$ remarkably reducing the storage demanding. 
Very diverse algorithms belong to this family 
such as projection methods
\cite{Simoncini2007},\cite{Druskin.Simoncini.11}, low-rank ADI \cite{Benner2009},\cite{Benner2014}, 
sign function methods \cite{Baur2008},\cite{Baur2006}. 
We refer the reader to \cite{Simoncini2016} for a full account of low-rank techniques.

Numerical methods for \eqref{main.eq} with large, banded, and not necessarily low rank $D$ 
have not been given attention so far, in spite of possible occurrence of this setting
in practical applications; see, e.g., \cite{Haber2016},\cite{Palitta2015},\cite{Jonsson.Kagstrom.02b}.

Our aim is to significantly contribute to this open problem
by introducing solution methods for generally banded data. In particular, a new general purpose
algorithm to handle an ill conditioned coefficient banded matrix $A$ is proposed.
 
If $A$ is well conditioned, the entries of $X$  present a decay
in absolute value as they move away from the banded pattern of $D$. Therefore,
 a banded approximation  %, that is a banded solution to (\ref{main.eq}), 
$\widehat X\approx X$ can be sought. This idea was exploited
in \cite{Haber2016}, where two algorithms for computing $\widehat X$ were proposed.
We show that if $A$ is well conditioned, a matrix-oriented formulation of the conjugate
gradient method ({\sc cg}) provides a quite satisfactory banded approximation 
at a competitive computational cost.
% a straightforward but very effective strategy for computing $\widehat X$.

%If no assumption on the condition number $\kappa(A)$ is made, 
For general symmetric banded data, the decay pattern of $X$ fades as the conditioning of $A$ worsens,
to the point that for ill-conditioned matrices, no appreciable (exponential) decay can be
detected in $X$.
 Nevertheless, we show that $X$ can be 
split into two terms, which can be well approximated 
 by a banded matrix and by a low-rank matrix, respectively. 
This observation leads to an efficient numerical procedure for solving 
 \eqref{main.eq}  both in terms of CPU time and memory requirements.
% To the best of our knowledge, no other method for \eqref{main.eq} is present in the literature unless strong assumptions
% on $\kappa(A)$ are considered.
 
 In principle one could apply the general purpose greedy 
algorithm proposed by Kressner and Sirkovi\'c in \cite{Kressner2015}.
%for generel linear matrix equations to the case of equation \eqref{main.eq}.
To be efficient, however, the method in  \cite{Kressner2015} requires that the solution $X$ admits a low-rank 
approximation; unfortunately,  this is not guaranteed in the case of a full-rank $D$.

Moreover, since data in \eqref{main.eq} are banded, they can be viewed as $\mathcal{H}$-matrices and the
algorithm derived in \cite{grasedyck-riccati} could 
be adapted for solving equation \eqref{main.eq}.  In this more 
general setting, the authors of \cite{grasedyck-riccati}
show that the solution $X$ to the Riccati equation (of which the Lyapunov equation is
the linear counterpart) can be well approximated by an $\mathcal{H}$-matrix, and 
 a sign function method equipped with $\mathcal{H}$-matrices arithmetic is proposed for its computation.
The application of this sophisticated procedure to the linear setting with simple banded structure appears
to be unnecessarily cumbersome. On the other hand,
algorithms directly applicable to banded matrices may be appealing
to practitioners, we thus refrain from implementing an ad-hoc version of the algorithm in 
\cite{grasedyck-riccati} for our purposes\footnote{We thank Lars Grasedyck for helpful remarks on the
topic.}.
% {\color{red}However, to the best of our knowledge, no Matlab implementation of the procedure is available on-line.
% Moreover, when dealing with differential equations, the discretization phase often leads to banded matrices and 
% application-oriented scientists like engineers, biologists, physicists, etc..., may be more comfortable in handling this 
% more intuitive structure instead of the very sophisticated $\mathcal{H}$-format.
% We thus think it is worth focusing on the special case of 
%banded matrices.}
 
The following is a synopsis of the paper. In section \ref{Well-conditioned} the matrix-oriented {\sc cg} method
is recalled and some of its sparsity pattern properties highlighted, to be used for $A$ well conditioned.
The case of ill-conditioned $A$ is addressed in section \ref{Ill_conditionedA}, while the detailed
procedure is illustrated in sections \ref{Itau}--\ref{IItau}. 
Section \ref{tau} discusses some crucial issues associated with parameter selections of the new method,
together with an automatic strategy for one of them. {The procedures presented 
in section~\ref{Well-conditioned} and section
\ref{Ill_conditionedA} are then generalized to the case of Sylvester equations
with banded data and symmetric positive definite coefficient matrices in section~\ref{Sylvester_eqs}.}
Results on our numerical experience are reported in section~\ref{Numerical_Examples}
 while our conclusions are given in section~\ref{Conclusions}.
 
 Throughout the paper we adopt the following notation. %Upper-case letters denote matrices while lower-case letters vectors.
 The $(i,j)$-th entry of the matrix $X$ is denoted by $(X)_{i,j}$ while $(x)_k$ is the $k$-th component of the vector $x$.
 Given a symmetric matrix $T$, $\beta_T$ denotes its bandwidth, that is $(T)_{i,j}=0$ for 
 $|i-j|>\beta_T$. For instance, if $T$ is tridiagonal, $\beta_T=1$. 
 If $T$ is symmetric, $\lambda_{\max}(T)$ and $\lambda_{\min}(T)$ 
are its largest and smallest eigenvalues, respectively. The matrix inner product is defined 
as $\langle X,Y\rangle_F:=\mbox{trace}(Y^TX)$ so that
  the induced norm is $\|X\|_F^2=\langle X,X\rangle_F$. The matrix 
norm induced by the Euclidean vector norm is denoted by $\|\cdot\|_2$ while we define
  $\|T\|_{\max}:=\max_{i,j}|(T)_{i,j}|$. Moreover, $\kappa(T)=\|T\|_2\|T^{-1}\|_2$ is the 
 spectral condition number of the invertible matrix $T$.
 % We define 
 % the operator $\mbox{vec}: \CC^{n \times n} \rightarrow \CC^{n^2}$ such that $\mbox{vec}(X)$ is the vector 
%obtained by stacking the columns of the matrix $X$ one on top of each other. 
The Kronecker product is denoted by $\otimes$
while $I_n$ denotes the identity matrix of order $n$, and $e_i$ its $i$-th column. The subscript 
is omitted whenever the dimension of $I$ is clear from the 
context. 
All our numerical experiments were performed in Matlab \cite{matlab7}.
%Finally, $e_i\in\mathbb{R}^n$ denotes the $i$-th vector of the canonical basis of $\mathbb{R}^n$.

%%%%%%%%%%%%%%%%%%%%%%%%%%%%%%%%%%%%%%%%%%%%%%%%%%%%%%%%%%%%%%%%%%%%%%%%%%%%%%%%%%%%%%%%%%%%%%%%%%%%%%%%%%%%%%%%%%%%%%%

\section{The case of well conditioned $A$}\label{Well-conditioned}
In the case when $A$ is well conditioned, it is possible to fully exploit the
banded structure of the data, and to substantially maintain it in a
suitably constructed approximate solution. To this end, advantage can be
taken of recently developed results on the entry decay
of functions of matrices, see, e.g., \cite{Benzi2007},\cite{Benzi2015},\cite{Canuto2014},\cite{Demko1984},
 by using the Kronecker form of the problem, that is %. %,\cite{Pozza2016}.
\begin{eqnarray}\label{eqn:kron}
{\cal A} x = {\rm vec}(D), \qquad x = {\rm vec}(X),\; \mathcal{A}:=A\otimes I+I\otimes A,
\end{eqnarray}
where ${\rm vec}(X)\in\mathbb{R}^{n^2}$ is the vector obtained by stacking the $n$ columns of $X$ one on top of each other.
Bounds for the entries of the inverse of $\mathcal{A}$ (viewed as a 
banded matrix with bandwidth $n\beta_A$) have been employed to estimate the decay in the entries of the solution 
$X$ to (\ref{main.eq}).
\begin{theorem}[\cite{Haber2016}]\label{decayTH}
 Consider equation (\ref{main.eq}).
% {\color{blue} and assume $A$ to be SPD.} 
Let 
 $$\tau:=\frac{1}{2|\lambda_{\max}(A)|}\max\left\{1,\frac{\left(1+\sqrt{\kappa(A)}\right)^2}{2\kappa(A)}\right\},
 \quad \mbox{and} \quad \rho:=\left(\frac{\sqrt{\kappa(A)}-1}{\sqrt{\kappa(A)}+1}\right)^{\frac{1}{n\beta_A}},$$
 %%
 %where $\lambda_{\mbox{max}}(A)$ is the largest eigenvalue of $A$ and $\kappa(A)$ its condition number.
 %If $x:=\mbox{vec}(X),d:=\mbox{vec}(D)$, then
 then the solution matrix $X$ satisfies 
 %%
%{\color{red}
 \begin{equation}\label{estdecay}
%|x_s|\leq \tau\sum_{i=1}^n\sum_{j=1}^n\left|D_{i,j}\right|\rho^{|(j-1)n+i-s|}.  
%|(x)_s|\leq \tau\sum_{t=1}^{n^2}\left|(d)_t\right|\rho^{|t-s|}.  
|(X)_{i,j}|\leq \tau\sum_{k=1}^n\sum_{\ell=1}^{n}\left|(D)_{k,\ell}\right|\rho^{|(\ell-j)n+k-i|}.  
\end{equation}
%}
 %%
\end{theorem}
By exploiting the Kronecker structure of $\mathcal{A}$,
sharper bounds for $\left(\mathcal{A}^{-1}\right)_{i,j}$ can be derived, see, e.g., \cite{Canuto2014},
leading to different, and possibly more accurate, estimates for $|(X)_{i,j}|$. 

\begin{theorem}\label{BoundTH2}
 Consider equation (\ref{main.eq}).
% {\color{blue} and assume $A$ to be SPD.} %and let $\lambda_{\mbox{max}}$, $\lambda_{\mbox{min}}$ respectively be the 
 %largest and the smallest eigenvalue of $A$. 
 Define $\lambda_1=\lambda_1(\omega):=\lambda_{\min}(A)+i\omega$, 
 $\lambda_2=\lambda_2(\omega):=\lambda_{\max}(A)+i\omega$, and $R:=\alpha+\sqrt{\alpha^2-1}$
 where $\alpha:=\left(|\lambda_1|+|\lambda_2|\right)/|\lambda_2-\lambda_1|$. 
 %For fixed $s,t\in\{1,\ldots, n^2\}$, we define
 %$m=\lfloor (s-1)/n\rfloor+1$, $\ell=s-n\lfloor (s-1)/n\rfloor$, $j=\lfloor (t-1)/n\rfloor+1$ and $i=t-n\lfloor (t-1)/n\rfloor$.
 %If $x=\mbox{vec}(X)$ and $d=\mbox{vec}(D)$,
 Then the solution matrix $X$ satisfies 
%%
%{\color{red}
 \begin{equation}\label{bound2}
 % |(x)_s|\leq \sum_{t=1}^{n^2}\theta_t|(d)_t|,
  |(X)_{i,j}|\leq \sum_{k=1}^n\sum_{\ell=1}^{n}\theta_{k,\ell}|(D)_{k,\ell}|,
  \end{equation}
%}
%%
where
\begin{itemize}
 %\item If $i\neq\ell$ and $j\neq m$, then
 \item If $k\neq i$ and $\ell\neq j$, then
 {\small
 $$
\theta_{k,l}=\frac{64}{2\pi|\lambda_{\max}(A)-\lambda_{\min}(A)|^2}
 \int_{-\infty}^{\infty}\left(\frac{R^2}{(R^2-1)^2}\right)^2
 \left(\frac{1}{R}\right)^{\frac{|k-i|}{\beta_A}+\frac{|\ell-j|}{\beta_A}-2} d\omega.
$$
 }
 %%
 %\item If either $i=\ell$ or $j=m$, then
 \item If either $k=i$ or $\ell=j$, then

 {\small
 $$
\!\!
\theta_{k,l}= \frac{8}{2\pi|\lambda_{\max}(A)-\lambda_{\min}(A)|}
 \int_{-\infty}^{\infty}\!\! \frac{1}{\sqrt{\lambda_{\min}(A)^2+\omega^2}}
 \frac{R^2}{(R^2-1)^2}\left(\frac{1}{R}\right)^{\frac{|k-i|}{\beta_A}+\frac{|\ell-j|}{\beta_A}-1}\!\! d\omega.
$$
 }
 %%
 
 %\item If both $i=\ell$ and $j=m$, then 
 \item If both $k=i$ and $\ell=j$, then 
 {\small
$$
\theta_{k,\ell}=\frac{1}{2\lambda_{\min}(A)}.
$$
 }
\end{itemize}
\end{theorem}

\begin{proof}
The statement directly comes from \cite[Theorem 3.3]{Simoncini2015} summing up on the entries of $D$.
\end{proof}

We emphasize that since $D$ is banded, only few $(D)_{k,\ell}$ are nonzero,
%$D$, i.e., few values $(d)_t$, are nonzero so that 
so that only few terms in the summation (\ref{bound2}) are actually computed.

\begin{figure}[htb]
        \centering
          \includegraphics[scale=0.55]{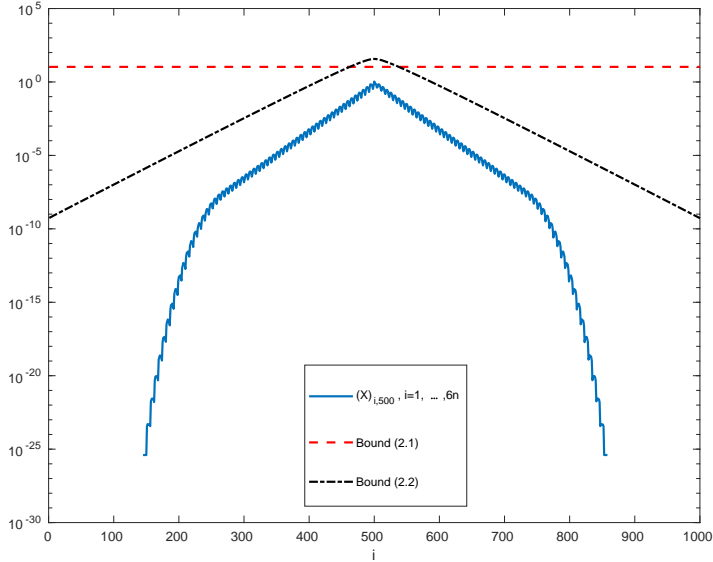}
          \caption{Example~\ref{Ex.1_bounds}. Magnitude of $(X)_{i,500}$, $i=1,\ldots,6n$,
          and its estimates (\ref{estdecay}) and (\ref{bound2}), with logarithmic scale. \label{Ex.1_Fig.1}}
 \end{figure}

\vskip 0.05in
\begin{num_example}\label{Ex.1_bounds}
{\rm
To illustrate the quality of the new bound compared with that in Theorem~\ref{decayTH} we consider
the data generated in Example~\ref{Ex.1} later in this section.
For $6n=1020$ in Figure~\ref{Ex.1_Fig.1} we report the magnitude in logarithmic scale of the entries of
the $500$-th column of the solution $X$, $\log_{10}(|X|_{i,500})$, $i=1,\ldots,6n$ (solid line), together with the 
corresponding
computed bounds in (\ref{estdecay}) (dashed line) and in (\ref{bound2}) (dashed and dotted line).
The new bound correctly captures the decay of the entries, while (\ref{estdecay}) predicts a
misleading almost flat slope. $\quad \square$
}
\end{num_example}

\vskip 0.05in

Since $A$ and ${\cal A}$ are both SPD, 
(\ref{main.eq}) can be solved by {\sc cg} applied to its Kronecker form (\ref{eqn:kron}).
In fact, the 
matrix-oriented {\sc cg} method can be implemented by directly employing $n\times n$ matrices,
 %in the numerical solution of (\ref{main.eq}), 
in agreement with similar matrix-oriented strategies in the literature;
see, e.g., \cite{Hochbruck1995} for an early presentation. %, and the references in \cite{Lin2013}. 

An implementation of the procedure is illustrated in Algorithm \ref{CG_matrix}. 
%%%
{\footnotesize
\begin{algorithm}
%\algsetup{linenosize=\small}
%\SetLine %% new algorithm2e: \SetAlgoLined
\caption{{\sc cg} for the Lyapunov matrix equation.\label{CG_matrix}}
\SetKwInOut{Input}{input}\SetKwInOut{Output}{output}
%%%%%%%%%%% INPUT %%%%%%%%%%%
\Input{$A\in\mathbb{R}^{n\times n},$ $A$ SPD, 
$D,X_0\in\mathbb{R}^{n\times n}$ with banded storage, $\epsilon_{res}>0$, $m_{\max}$}
%%%%%%%%%%% OUTPUT %%%%%%%%%%%
\Output{$X_k\in\mathbb{R}^{n\times n}$}
%%%%%%%%%%%%%%%%%%%%%%%%%%%%%%%%%%%
\BlankLine
\nl Set $R_0=D-AX_0-X_0A$, $P_0=R_0$ \\
\For{$k = 1,2,\dots, m_{\max}$}{
\nl $W_k=AP_{k-1}+P_{k-1}A$	\\
\nl $\alpha_k=\frac{\|R_{k-1}\|_F^2}{\langle P_{k-1},W_k\rangle_F}$ \\
\nl $X_k=X_{k-1}+P_{k-1}\alpha_k$\\
\nl $R_k=R_{k-1}-W_k\alpha_k$\\
\nl \If{$\|R_k\|_F/\|R_0\|_F<\epsilon_{res}$}{ 
\nl \textbf{Stop} \\ }
\nl $\beta_k=\frac{\|R_k\|_F^2}{\|R_{k-1}\|_F^2}$\\
\nl $P_k=R_k+P_{k-1}\beta_k$\\
}
\end{algorithm}
}
%%%

Several properties of Algorithm \ref{CG_matrix} can be observed. For instance, since $D$ is symmetric, it is easy to show 
that all the iterates, $W_k,X_k,P_k,R_k$, are symmetric for all $k$ if a symmetric $X_0$ is chosen.
This implies that only one matrix-matrix multiplication by $A$ in line 
2 is needed. Indeed, if $S_k:=AP_{k-1}$, then $W_k=AP_{k-1}+P_{k-1}A=AP_{k-1}+(AP_{k-1})^T=S_k+S_k^T$.
Furthermore, only the lower -- or upper -- triangular part of the iterates need 
 to be stored, leading to some gain in 
terms of both memory requirements and number of floating point operations (flops).
Various algebraic simplifications can be
implemented for the matrix inner products and the Frobenius norms in lines 3, 6, 8 as 
well as for the matrix-matrix products in line 2.

We next show that all the matrices involved in Algorithm \ref{CG_matrix} are banded matrices, with
bandwidth linearly depending on $k$, the number of iterations performed so far.
This matrix-oriented procedure is effective in maintaining
the banded structure as long as $k$ is moderate, and this is related to the
conditioning of the coefficient matrix.
 %We show that all the matrices involved in the algorithm, and thus the approximate solution $X_k$, are banded matrices.

 \begin{proposition}\label{THM1}
 If $X_0=0$,  
  all the iterates generated by Algorithm \ref{CG_matrix} are banded matrices and, in particular,
\begin{eqnarray*}
   \beta_{W_k}\leq k\beta_{A}+\beta_{D}, \quad \beta_{X_k}\leq (k-1)\beta_{A}+\beta_{D}, \quad %\\
   \beta_{R_k}\leq k\beta_{A}+\beta_{D}, \quad  \beta_{P_k}\leq k\beta_{A}+\beta_{D}.
\end{eqnarray*}
 \end{proposition}

 \begin{proof}
%{\it Proof.}
 We first focus on the effects of Algorithm~\ref{CG_matrix} on the bandwidth of the current iterates.
 We recall that if $G,H\in\mathbb{R}^{n\times n}$ are banded matrices with bandwidth $\beta_G,\beta_H$ respectively,
 the matrix $GH$ has bandwidth at most $\beta_G+\beta_H$. 
 The multiplication by $A$ in line 2 of Algorithm \ref{CG_matrix} is the only step that increases
 the iterate bandwidth at iteration $k$, therefore we have
  $\beta_{W_k}\leq \beta_{A}+\beta_{P_{k-1}}$, $\beta_{X_k}\leq \max\{\beta_{X_{k-1}},\beta_{P_{k-1}}\}$,
  $\beta_{R_k}\leq \max\{\beta_{R_{k-1}},\beta_{W_k}\}$ and $\beta_{P_k}\leq \max\{\beta_{R_k},\beta_{P_{k-1}}\}$.
 %%
 %$$\begin{array}{ll}
 %  \beta_{W_k}\leq \beta_{A}+\beta_{P_{k-1}},& \beta_{X_k}\leq \max\{\beta_{X_{k-1}},\beta_{P_{k-1}}\},\\
 %  &\\
 %  \beta_{R_k}\leq \max\{\beta_{R_{k-1}},\beta_{W_k}\},& \beta_{P_k}\leq \max\{\beta_{R_k},\beta_{P_{k-1}}\}.\\
 % \end{array}$$
  %%
 We now demonstrate the statement by induction on $k$. Since $X_0=0$, $R_0=D$ and 
 $\beta_{R_{0}}=\beta_{P_{0}}=\beta_{D}$. Moreover, for $k=1$,
 $$
\begin{array}{ll}
   \beta_{W_1}\leq \beta_{A}+\beta_{D},%\\
   & \beta_{R_1}\leq \max\{\beta_{R_{0}},\beta_{W_1}\}\leq\beta_{A}+\beta_{D},\\
   \\
   \beta_{X_1}=\beta_{D},
   &%\\
   \beta_{P_1}\leq \max\{\beta_{R_1},\beta_{P_{0}}\}\leq\beta_{A}+\beta_{D}.\\
  \end{array}
$$
 Supposing that the statement holds for $k=j-1>1$, we prove it for $k=j$.
 $$
\begin{array}{l}
   \beta_{W_{j}}\leq \beta_{A}+\beta_{P_{j-1}}\leq\beta_{A}+(j-1)\beta_{A}+\beta_{D}=j\beta_{A}+\beta_{D},\\
  \\
  \beta_{X_j}\leq\max(\beta_{X_{j-1}},\beta_{P_{j-1}})\leq\beta_{P_{j-1}}\leq(j-1)\beta_{A}+\beta_{D},\\
   \\
   \beta_{R_j}\leq\max(\beta_{R_{j-1}},\beta_{W_j})\leq\beta_{W_j}\leq j\beta_{A}+\beta_{D},\\
   \\
   \beta_{P_j}\leq\max(\beta_{R_j},\beta_{P_{j-1}})\leq \beta_{R_j}\leq  j\beta_{A}+\beta_{D}. \qquad \square
  \end{array}
$$
\end{proof}

A similar result can be shown if $X_0$ is a banded matrix.
Theorem~\ref{THM1} implies that after $k$ iterations
%as long as the number of iterations
% $k$ is moderate, all the {\sc cg} iterates can be cheaply stored.
%Indeed, at iteration $k$, 
all iterates are banded matrices with bandwidth at most $k\beta_A+\beta_D$.
Moreover, only their lower (or upper) triangular parts are stored so that the number 
of nonzero entries of each iterate is at most
%%
%{\small
$$
\displaystyle{n+\sum_{i=1}^{k\beta_A+\beta_D}(n-i)}=n+(k\beta_A+\beta_D)n-\frac{1}{2}(k\beta_A+\beta_D)(k\beta_A+\beta_D-1)
=\mathcal{O}(n).
$$
%}
%%
Exploiting Theorem~\ref{THM1}, it can be shown that the computational cost of Algorithm \ref{CG_matrix} linearly
scales with the problem size $n$. This is a major saving of the matrix-oriented version of the algorithm,
compared with its standard vector-oriented counterpart with ${\cal A}$, which would require ${\cal O}(n^2)$ operations
per iteration.

\begin{corollary}\label{COR1}
For small values of $k$, the computational cost of the $k$-th iteration of Algorithm \ref{CG_matrix} amounts 
to $\mathcal{O}(n)\,flops$.
\end{corollary}

\begin{proof}
 We first notice that if $G,H\in\mathbb{R}^{n\times n}$ are banded matrices with bandwidth $\beta_G,\beta_H$ respectively,
 the matrix-matrix product $GH$ costs $\mathcal{O}\left(n(2\beta_G+1)(2\beta_H+1)\right)$ flops.
Therefore, the number of operations required by line 2 of Algorithm \ref{CG_matrix} is
$$\mathcal{O}\left(2n(2\beta_A+1)(2\beta_{P_{k+1}}+1)\right)=\mathcal{O}\left(2n(2\beta_A+1)(2(k\beta_{A}+\beta_D)+1)\right)
=\mathcal{O}\left(8k\beta_A^2n\right).
$$
Similarly, matrix-matrix products with banded matrices determine
the matrix inner products $\langle \cdot,\cdot\rangle_F$, and thus the Frobenius norms $\|\cdot\|_F$,
in lines 3 and 8.  %, matrix-matrix products with banded matrices are required.
Finally, again the summations in lines 4,5 and 9 require a number of operations of the order of 
the number of nonzero entries
of the matrices involved, that is $\mathcal{O}(n)$. % thanks to the observation above.
\end{proof}

For a given tolerance, we can predict the number of iterations required by {\sc cg} to converge and thus
the bandwidth of the computed numerical solution.
To this end, classical {\sc cg} convergence
results (see, e.g., \cite[Section 13.2.1]{Axelsson1994}) can be applied. 
%We report the following statement and its proof for sake of completeness.

\begin{theorem}[{\cite[Equation (13.12)\footnote{Since $\mathcal{A}$ is SPD, $\mathcal{K}(S)=\rho=K=1$ in
\cite[Equation (13.12)]{Axelsson1994}.}]{Axelsson1994}}]
Let $err_j:=\|\mbox{vec}(X_*)-\mbox{vec}(X_{j})\|_{\mathcal{A}}$ be the error in the energy norm associated with 
 the exact solution $X_*$ to (\ref{main.eq}). Moreover,  let
$\sigma=\frac{1-\sqrt{\kappa(\mathcal{A})^{-1}}}{1+\sqrt{\kappa(\mathcal{A})^{-1}}}$.
 Then, for a given tolerance $\epsilon_{res}$,
 the matrix $X_{\bar k}$ computed by performing $\bar k$ iterations
 of Algorithm~\ref{CG_matrix}, with
 \begin{equation}\label{bar_k}
 \bar k=\lceil \log\left(\epsilon_{res}^{-1}+\sqrt{\epsilon_{res}^{-2}-1}\right)/\log(\sigma^{-1})\rceil,
 \end{equation}
 %%
%where $\sigma=\frac{1-\sqrt{\kappa(\mathcal{A})^{-1}}}{1+\sqrt{\kappa(\mathcal{A})^{-1}}}$, 
is such that 
$\frac{err_{\bar k}}{err_0}\leq  \epsilon_{res}$. 
\end{theorem}

\begin{Cor}\label{Cor}
 With the notation above, it holds that 
$\sigma=\frac{1-\sqrt{\kappa(A)^{-1}}}{1+\sqrt{\kappa(A)^{-1}}}$.
 Moreover, for $\bar k$ as in (\ref{bar_k}),  $X_{\bar k}$ is banded 
with bandwidth $\beta_{X_{\bar k}}\leq (\bar k-1)\beta_{A}+\beta_{D}$.
 \end{Cor}

 \begin{proof}
  %Applying Algorithm \ref{CG_matrix} to (\ref{main.eq}) is mathematically equivalent to solving the linear system
  %%
  %\begin{equation}\label{linear.system}
  %\mathcal{A}\mbox{vec}(X)=\mbox{vec}(D), 
 % \end{equation}
  %%
  %where $\mathcal{A}=A\otimes I+I\otimes A$, by the ``standard'' {\sc cg} procedure. Therefore, the results on the 
  %convergence of the latter can be employed in the study of the solution of (\ref{main.eq}) by Algorithm \ref{CG_matrix}.
  %It holds
  %%
  %$$
 %\frac{err_j}{err_0}=\frac{\|\mbox{vec}(X_*)-\mbox{vec}(X_{j})\|_{\mathcal{A}}}{\|\mbox{vec}(X_*)-\mbox{vec}(X_0)\|_{\mathcal{A}}}\leq
 %2\left(\frac{\sqrt{\kappa(\mathcal{A})}-1}{\sqrt{\kappa(\mathcal{A})}+1}\right)^{j},$$
%%
% From \eqref{bound_cg}, a direct computation shows that $\bar k$ iterations,
% %%
%  $$\bar k\geq \frac{\log\frac{\epsilon_{res}}{2}}{\log(\sqrt{\kappa_2(\mathcal{A})}-1)-\log(\sqrt{\kappa_2(\mathcal{A})}+1)},$$
% %%
% ensures
%  %%
%  $\frac{err_{\bar k}}{err_0}\leq \epsilon_{res}.$
%  %%
% Moreover, 
Both $A$ and $\mathcal{A}$ are SPD. Moreover, it is well known that
$\kappa(\mathcal{A})=\frac{\lambda_{\max}(\mathcal{A})}{\lambda_{\min}(\mathcal{A})}=
\frac{2\lambda_{\max}(A)}{2\lambda_{\min}(A)}=\kappa(A)$.
%%
%and the matrix $\mathcal{A}$ does not need to be assembled to estimate $\bar k$.
%
The result follows by recalling from Theorem \ref{THM1} that $X_{\bar k}$ is a banded matrix such that
$\beta_{X_{\bar k}}\leq (\bar k-1)\beta_{A}+\beta_{D}$.
\end{proof}

When $A$ is well conditioned, the simple matrix-oriented {\sc cg} typically outperforms
%the Algorithm \ref{CG_matrix} can be employed to efficiently solve equation (\ref{main.eq}).
%This is a very straightforward strategy but it outperforms 
more sophisticated methods proposed in the very recent literature. A typical situation is
reported in the next example.

\begin{num_example}\label{Ex.1}
{\rm
 We consider an example from \cite{Haber2016}, where  $A=M\otimes I_6+I_n\otimes L\in\mathbb{R}^{6n\times 6n}$,
 $M=\mbox{tridiag}(e,\underline{e},e)\in\mathbb{R}^{n\times n}$, 
 $L=\mbox{tridiag}(e,\underline{a-e},e)\in\mathbb{R}^{6\times 6}$, $e=-0.34$, $a=1.36$.
 The right-hand side is $D=Q\otimes\mathbf{1}\mathbf{1}^T+0.8I_{6n}$ where $\mathbf{1}\in\mathbb{R}^6$
 is the vector of all ones and $Q=\mbox{tridiag}(0.1,\underline{0.2},0.1)\in\mathbb{R}^{n\times n}$;
note the change of sign in $A$ and $D$ compared with \cite{Haber2016}.
%This problem is equivalent to Example 1 in \cite{Haber2016}; we replaced the matrices $A$ and $D$ of \cite{Haber2016}
%with $-A$ and $-D$ to make the coefficient matrix positive definite.
% 
 Both matrices $A$ and $D$ are block tridiagonal with blocks of size $6$ and $\beta_A=6$, $\beta_D=11$. 
 Thanks to the Kronecker structure of $A$, it is easy to provide an estimate of its condition number which 
 turns out to be independent of $n$ as
 $\lambda_{\max}(A)=\lambda_{\max}(M)+\lambda_{\max}(L)$ and 
 $\lambda_{\min}(A)=\lambda_{\min}(M)+\lambda_{\min}(L)$.
Since $M$ and $L$ are tridiagonal Toeplitz matrices, we can explicitly compute their spectrum:
 %We can explicitly compute the spectrum of 
% $L\in\mathbb{R}^{6\times 6}$ as it is fixed. In particular, 
$\lambda_{\max}(L)=a-e+2|e|\cos(\frac{\pi}{7})$, $\lambda_{\min}(L)=a-e+2|e|\cos(\frac{6}{7}\pi)$,
$\lambda_{\max}(M)=e+2|e|\cos(\frac{\pi}{n+1})$ and $\lambda_{\min}(M)=e+2|e|\cos(\frac{n}{n+1}\pi)$;
see, e.g., \cite{Smith1978}.
%  The eigenvalues location of $M$ can instead be determined by applying 
% Gerschgorin's theorems, since the structure of the matrix
%  does not change when $n$ increases: it holds that $\lambda_{\max}(M)\leq -e=0.34$ and 
%  $\lambda_{\min}(M)\geq 3e=-1.02$.
 Therefore,% $\lambda_{\max}(A)\leq 0.34+2.3127=2.6527$, $\lambda_{\min}(A)\geq -1.02+1.0873=0.0673$ and
 \begin{eqnarray*}
 \kappa(A)&=&\frac{\lambda_{\max}(A)}{\lambda_{\min}(A)}
 =\frac{a+2|e|\left(\cos(\frac{\pi}{7})+\cos(\frac{\pi}{n+1})\right)}{a+2|e|\left(\cos(\frac{6}{7}\pi)+
 \cos(\frac{n}{n+1}\pi)\right)}=\frac{a+2|e|\left(\cos(\frac{\pi}{7})+\cos(\frac{\pi}{n+1})\right)}{a-
 2|e|\left(\cos(\frac{\pi}{7})+\cos(\frac{\pi}{n+1})\right)}\\
 &\leq&  \frac{a+2|e|\left(\cos(\frac{\pi}{7})+1\right)}{a- 2|e|\left(\cos(\frac{\pi}{7})+1\right)}\leq 40, \quad \mbox{for all }n. 
 \end{eqnarray*}
 The matrix $A$ is thus well-conditioned and Algorithm \ref{CG_matrix} can be employed in the solution process.
 %By using classical {\sc cg} convergence
%results (see, e.g., \cite[Section 13.2.1]{Axelsson1994}),
 By \eqref{bar_k}, it follows that $\bar k=45$ iterations will be sufficient to obtain
a relative error (in the energy norm) less than $10^{-6}$ for all $n$.
% $err_{\bar k}/err_0<\epsilon_{CG}=10^{-6}$ for all $n$. 
The solution $X_{\bar k}$ will be a banded 
 matrix with bandwidth $\beta_{X_{\bar k}}\leq 44\cdot\beta_A+\beta_D=275$.
 
 We next apply Algorithm \ref{CG_matrix} for different values of $n$ and relative residual tolerance $10^{-6}$,
and we compare the method performance with that of the second procedure described in \cite{Haber2016}.
 This method consists of a gradient projection method applied to $\min_{X}\|D-AX-XA\|_F^2$
 where the initial guess is chosen as a coarse approximation to the integral in \eqref{Xintegral}. We employ 
 the same setting suggested by the authors; see \cite{Haber2016} for details.
 %In particular,
 %following the notation in \cite{Haber2016}, $q=60$, $M=20$ and the maximum number of iteration for the gradient 
 %projection method is $50$.
 The results are collected in Table~\ref{Ex.1_Tab.1} where the CPU time is expressed in seconds.
In the first instance, Algorithm \ref{CG_matrix} is stopped as soon as the relative residual norm satisfies the
stopping criterion. In the second instance, a fixed number of iterations for Algorithm \ref{CG_matrix}
is used, so as to obtain the same final approximate solution bandwidth  as that of the 
procedure in \cite{Haber2016}.
With this second instance, we are able to directly compare the accuracy and efficiency of
{\sc cg} and of the method in \cite{Haber2016}.
 
%% 
%\begin{table}[!ht]
%\centering
%%{\small
% \begin{tabular}{r|rrrr|rrr}
% $6n$ & \multicolumn{4}{c|}{{\sc cg} (Algorithm \ref{CG_matrix})} & \multicolumn{3}{c}{Algorithm \cite{Haber2016} } \\
%  &  Its. & $\beta_X$  &Time  &  Res.  &  $\beta_X$ &Time &  Res.   \\
%  %\hline
%  \hline 
%    $10200$ &  45    & 275 & 17.12 &  8.44e-7 & 53 & 123.14 &  5.55e-1\\
%  $102000$ &  45   & 275 &170.88   & 8.44e-7 & 53 &1880.29 & 5.55e-1  \\
%  $1020000$ &  45  &275 & 1677.27  & 8.44e-7 & 53 &23822.94  & 5.55e-1 \\
% \end{tabular}
% \caption{Algorithm \ref{CG_matrix} and the second procedure presented in \cite{Haber2016} applied 
%to Example~\ref{Ex.1}.
% Results for different values of $6n$. \label{Ex.1_Tab.1}}
% \end{table}
%
%% 
\begin{table}[!ht]
\centering
%{\small
{\footnotesize
 \begin{tabular}{r|rrrr|rrrr|rrr}
 $6n$ & \multicolumn{4}{c|}{{\sc cg} (Algorithm \ref{CG_matrix})} & 
\multicolumn{4}{c|}{{\sc cg} (Algorithm \ref{CG_matrix})} &
\multicolumn{3}{c}{Algorithm \cite{Haber2016} } \\
  &  Its. & $\beta_X$  &Time  &  {\bf Res.}  &  Its. & ${\pmb\beta}_X$  &Time  &  Res.  &  
$\beta_X$ &Time &  Res.   \\
  %\hline
  \hline
 $10200$ &  45    & 275 & 17.1 &  8.4e-7 &  8    & 53 & 0.7 &  1.2e-1 & 53 & 123.1 &  5.5e-1\\
 $102000$ &  45   & 275 &170.8   & 8.4e-7 &  8   & 53 &4.6   & 1.2e-1  & 53 &1880.2 & 5.5e-1  \\
 $1020000$ &  45  &275 & 1677.2  & 8.4e-7 &  8  &53 & 56.9  & 1.2e-1 & 53 &23822.9  & 5.5e-1 \\
 \end{tabular}
}
 \caption{Algorithm \ref{CG_matrix} and the second procedure presented in \cite{Haber2016} applied
to Example~\ref{Ex.1}.
 Results are for different values of $6n$. For {\sc cg},
in bold is the quantity used in the stopping criterion. \label{Ex.1_Tab.1}}
 \end{table}

{Because the condition number is bounded independently of $n$, the number of {\sc cg} iterations is
also bounded by a constant independent of $n$}; this is clearly shown in the table. Therefore the
total CPU time to satisfy a fixed convergence criterion scales linearly with $n$.
The results illustrated in Table~\ref{Ex.1_Tab.1} show that Algorithm \ref{CG_matrix} is very effective,
 in terms of CPU time, while it always reaches the desired residual norm, when this is used as
stopping criterion. This is not the case for the
algorithm in \cite{Haber2016}, which would probably require a finer parameter tuning to be able
to meet all stopping criteria. 
%On the other hand, algorithm in \cite{Haber2016} provides a solution $X$ with a much thinner bandwidth compared to the 
%one produced by CG. For sake of completeness, in Table \ref{Ex.1_Tab.2}, we also report the CG results when a $X$ such that 
%$\beta_X=53$ is computed.
%\begin{table}[!ht]
%\centering
%{\small
% \begin{tabular}{r|rrrr}
% $6n$ &  Its. & $\beta_X$  &Time  &  Res.   \\
%  %\hline
%  \hline 
%    $10200$ &  8    & 53 & 0.71 &  1.22e-1 \\
%  $102000$ &  8   & 53 &4.63   & 1.22e-1  \\
%  $1020000$ &  8  &53 & 56.94  & 1.22e-1  \\
% \end{tabular}
% \caption{Algorithm \ref{CG_matrix} for computing $X$ such that $\beta_X=53$.
% Results for different values of $6n$. \label{Ex.1_Tab.2}}
%%\quad
% \end{table}
%%
If the final bandwidth is the stopping criterion,
the obtained accuracy is comparable with the results of the algorithm in \cite{Haber2016}, however 
{\sc cg} is many orders of magnitude faster.

We next compare the memory-saving solution $X_k$ computed by the {\sc cg} algorithm to the dense 
solution $X$ obtained with 
the Bartels-Stewart method \cite{Bartels1972} implemented in Matlab as the function {\tt lyap}.
To this end, we consider a small problem, $6n=1020$ and set $\epsilon_{res}=10^{-6}$.
{\sc cg} converges in $k=45$ iterations providing a solution $X_k$ 
such that $\beta_{X_k}=275$.
In Figure~\ref{contour} we plot in logarithmic scale the relative magnitude of the entries of 
$X_k-\widebar X$ where $\widebar X$ is obtained from $X$ by retaining only its first 275 (upper and lower)
diagonals.

\begin{figure}[htb]
        \centering
          \includegraphics[scale=0.7]{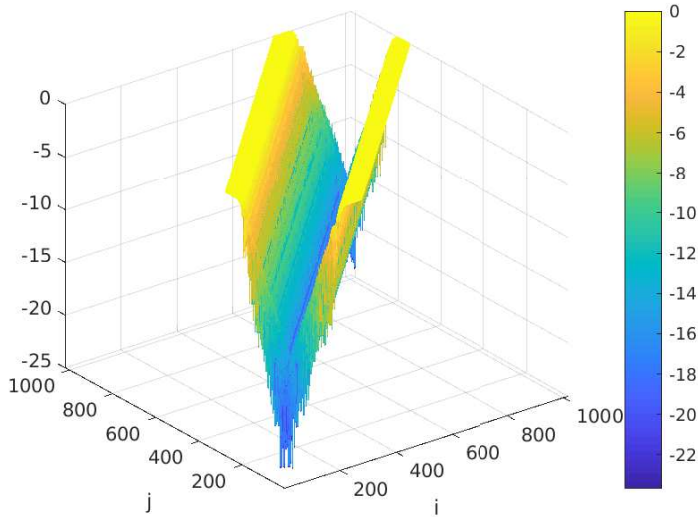}
          \caption{Decay pattern of the entrywise relative error of the {\sc cg} approximate
solution matrix (Logarithmic scale). \label{contour}}
         %\caption{$\mathtt{mesh(log(abs(X_k-\bar X)./abs(\bar X)))}$. \label{contour}}
 \end{figure}

As expected, the error in the approximate solution $X_k$ is concentrated in the entries of the most external 
diagonals.  %, which are the most recently computed ones. % to have been computed.
Indeed, due to the decay pattern of $X$, the largest entries of $X$
are gathered near the main diagonal, and these must be well approximated to reach the prescribed 
accuracy. %Moreover, the components close to the main diagonal seem to be rather precise. 
Intuitively, the corresponding entries of $X_k$ have been 
refined as the iterations proceed so that they do not contribute to the entry-wise error. 
}
%\endproof 
$\qquad \square$

\end{num_example}

\vskip 0.05in

We recommend using the matrix-oriented {\sc cg} method for well-conditioned $A$, while
for ill-conditioned problems we present a new method in the next section.
Nonetheless, in case of moderately ill-conditioned $A$, one may still want to employ {\sc cg} and
%want to 
apply a preconditioning
operator $\mathcal{P}$ to further reduce the number of {\sc cg} iterations. %needed to converge. 
%
%At iteration $k$, the preconditioning step $\mathcal{P}^{-1}(P_k)$ is performed in the loop of Algorithm~\ref{CG_matrix}
%before the computation of $W_k$ in line 2.
However, the derivation of such a %an efficient preconditioner operator 
$\mathcal{P}$ is not straightforward in our context.
In addition to reducing the iteration count at low cost, the application of $\mathcal{P}$
should preserve the banded structure of the subsequent iterates.
This is surely an interesting problem to explore, however it goes beyond the
scope of this work.
%Clearly, $\mathcal{P}$ must be effective in reducing the iteration count but the energy and memory-saving 
%features of Algorithm~\ref{CG_matrix} have to be preserved.
%In particular, only $\mathcal{O}(n)$ flops must be required to perform the step $\mathcal{P}^{-1}(P_k)$ and
%the preconditioner must maintain the banded structure of the iterates.
%All these requirements are often very tricky to be guaranteed so that we suggest to employ the unpreconditioned 
%{\sc cg} algorithm presented in Algorithm~\ref{CG_matrix}.}

The situation changes significantly if $A$ is ill conditioned, since a larger number of
iterations will be required to determine a sufficiently good approximation. 
This difficulty is not a peculiarity of the method, but rather
it reflects the fact that the exact solution $X$ cannot be well represented by a
banded matrix. Therefore, any acceleration strategy to reduce the {\sc cg} iteration count will
necessarily end up constructing a denser approximation. 
In this case, a different strategy needs to
be devised, and this is discussed in the next section.

%%%%%%%%%%%%%%%%%%%%%%%%%%%%%%%%%%%%%%%%%%%%%%%%%%%%%%%%%%%%%%%%%%%%%%%%%%%%%%%%%%%%%%%%%%%%%%%%%%%%%%%%%%%%%%%%%%%%%%%%%%%
 \section{A new method for ill-conditioned $A$}\label{Ill_conditionedA}
If $A$ is ill-conditioned, the entries of the solution $X$ to (\ref{main.eq}) do not have, in general,
a fast decay away from the diagonal, so that a banded approximation is usually not sufficiently accurate.
By using the following closed form for the matrix $X$ (see, e.g., \cite{Lancaster1970})
\begin{equation}\label{Xintegral}
 X=\int_{0}^{+\infty}e^{-tA}De^{-tA}dt,
\end{equation}
we next derive a splitting of the matrix $X$ that leads to a memory saving
approximation.
 The simple proof is reported for sake of completeness, as the result without
proof is stated by Kailath as an exercise\footnote{We thank a referee for citing an article
pointing to Kailath's book for this result.} in \cite[Exercise 2.6-1]{Kailath1980}.
%This result has been also exploited in \cite{Gawronski1990} to compute gramians over time limited intervals.}

\begin{theorem}\label{th:splitting}
Let $X(\tau) = \int_{0}^{\tau}e^{-tA}De^{-tA}dt$, for $\tau >0$, 
so that $X \equiv X(+\infty)$.
For $\tau>0$ the matrix $X$ in (\ref{Xintegral})  can be written as
 \begin{equation}\label{eqn:splitting}
  X= X(\tau) + e^{-\tau A} X e^{-\tau A} .
 \end{equation}
\end{theorem}
%%
%{\it Proof.}
\begin{proof}
We can split $X$ as  $X = \int_0^\tau e^{-tA}De^{-tA}dt + \int_\tau^{+\infty}e^{-tA}De^{-tA}dt$, 
where the first term is $X(\tau)$.
Performing the change of variable $t=s+\tau$ it holds
\begin{eqnarray*}
\int_\tau^{+\infty}e^{-tA}De^{-tA}dt
&=&\int_0^{+\infty}e^{-(s+\tau)A}De^{-(s+\tau)A}ds\\
&=&e^{-\tau A}\int_0^{+\infty}e^{-sA}De^{-sA}dse^{-\tau A} =e^{-\tau A}Xe^{-\tau A}. \qquad \square
\end{eqnarray*}
\end{proof}
%\vskip 0.1in

The splitting in (\ref{eqn:splitting}) emphasizes two terms in the solution matrix $X$. If
$\tau$ is sufficiently large and the eigenvalues of $A$ present a global decay,
 the second term is clearly numerically low rank, since $e^{-\tau A}$ is 
numerically low rank.
Depending on the magnitude of $\tau A$, the following Theorem~\ref{Th:decay_exp} proved 
in \cite{Benzi2015} ensures
that the first term is banded. As a result, Theorem~\ref{th:splitting} provides a splitting
of $X$ between its banded and numerically low rank parts. Our new method aims at approximating
these two terms separately, so as to limit memory consumption.
\vskip 0.1in
\begin{theorem}[\cite{Benzi2015}]\label{Th:decay_exp}
 Let $M$ be Hermitian positive semidefinite with eigenvalues in the interval $[0,4\rho]$. Assume in addition that $M$ is 
 $\beta_M$-banded. For $k\neq \ell$, let $\xi=\lceil |k-\ell|/\beta_M\rceil$, then
 \begin{itemize}
  \item[(i)] For $\rho t\geq 1$ and $\sqrt{4\rho t}\leq \xi\leq 2\rho t$,
  $|(e^{-tM})_{k,\ell}|\leq 10\,e^{-\frac{\xi^2}{5\rho t}};$
  \item[(ii)] For $\xi\geq 2\rho t$,
  $|(e^{-tM})_{k,\ell}|\leq 10\frac{e^{-\rho t}}{\rho t}\left(\frac{e\rho t}{\xi}\right)^\xi.$
 \end{itemize}
\end{theorem}

\vskip 0.1in
In our setting, Theorem~\ref{Th:decay_exp} can be applied to $e^{-t(A-\lambda_{min}I)}$
by appropriately scaling the original matrix $e^{-tA}$.
%introducing the factor $e^{-t\lambda_{min}}$ in the above estimates.
For small $t$, Theorem \ref{Th:decay_exp} ensures that $e^{-tA}$ has small components
away from the diagonal so that it can be well approximated by a banded matrix, $\reallywidehat{e^{-tA}} \approx e^{-tA}$; 
the product $\reallywidehat{e^{-tA}} D \reallywidehat{e^{-tA}}$ is still banded.

With these considerations in mind, we are going to approximate $X$ by
estimating the two quantities $X(\tau)$,  $e^{-\tau A} X e^{-\tau A}$ in (\ref{eqn:splitting}),
for a suitable $\tau>0$, that is
$$
  X= X(\tau) + e^{-\tau A} X e^{-\tau A} \approx X_B + X_L,
$$
where the banded matrix $X_B$ approximates the fast decaying portion $X(\tau)$, while
$X_L$ approximates the numerically low rank part $e^{-\tau A} X e^{-\tau A}$.

%%%%%%%%%%%%%%%%%%%%%%%%%%%%%%%%%%%%%%%%%%%%%%%%%%%%%%%%%%%%%%%%%%%%%%%%%%%%%%%%%%%%%%%%%%%%%%%%%%%%%%%%%%%%%%%%%%%%%%%%
\subsection{Approximating $X(\tau)$ by a banded matrix}\label{Itau}
The approximation of the first term by a banded matrix is obtained with the following steps:

i) We first replace the integral in $X(\tau)$ by an adaptive quadrature formula;

ii) We approximate the two exponential matrix functions by rational counterparts,
using a partial fraction expansion;

iii) We truncate the elementary terms in the partial fraction expansion to banded form.

The a-priori accuracy of the first two steps can be estimated by using well established
results in the literature applied to the eigendecomposition of $A$. 
In the third step, terms of the type $(t_i A - \xi_jI)^{-1}$ are dense, however recent theoretical
results ensure that they can be approximated with banded matrices by truncation.
%will be truncated to form the final banded approximation, by using recently developed results.

We start with step (i), that is
%We first focus on $\mathcal{I}_\tau$ that can be approximated by a quadrature formula,
%%
\begin{equation}\label{gauss_lobatto}
 X(\tau)=\int_{0}^{\tau}e^{-tA}De^{-tA}dt\approx \frac{\tau}{2}\sum_{i=1}^\ell\omega_i e^{-t_iA}De^{-t_iA},
\end{equation}
where $t_i=\frac{\tau}{2}x_i+\frac{\tau}{2}$, while $x_i,\omega_i$ are respectively the 
nodes and weights of the formula; in our experiments we considered
a matrix-oriented version of the adaptive
Gauss-Lobatto quadrature in \cite[Section~4.5]{Gander2000} with given tolerance $\epsilon_{quad}$.

As for step (ii), rational functions provide very accurate approximations to the matrix exponential
$e^{-A} \approx \mathcal{R}_\nu(-A)$. See, e.g., \cite{Baker1996},\cite{Carpenter1984},\cite{Trefethen2006}. 
In our setting rational Chebyshev functions in $\RR^+$ appear to be appropriate. They
admit the following partial fraction expansion
\begin{equation}\label{Cheb_ratfun}
\mathcal{R}_\nu(A)=\sum_{j=1}^\nu\theta_j(A-\xi_jI)^{-1}, 
\end{equation}
where $\theta_j,\xi_j\in\mathbb{C}$ are its weights and (distinct) poles, respectively. 
For $A$ real, the poles $\xi_j$ are complex conjugate, yielding  the simplified form %(see, e.g., \cite{ADDREF})
%For $A$ real, the poles $\xi_j$ are complex conjugate, so that we can write the simplified form (see, e.g., \cite{ADDREF})
% are complex conjugate and the computation of some inverses in (\ref{Cheb_ratfun})
%can be avoided. Indeed, we can write
%%
\begin{equation}\label{reducing_nu}
 \mathcal{R}_\nu(A)=\sum_{\substack{{j=1,}\\{j \; odd}}}^{\nu-1}2\mbox{Re}\left(\theta_j\left(t_iA-\xi_jI\right)^{-1}\right)+
 \theta_\nu\left(t_iA-\xi_\nu I\right)^{-1}, 
\end{equation}
where $\xi_\nu$ is the real pole of $\mathcal{R}_\nu$ if $\nu$ is odd. % NOT APPROPRIATE See, e.g., \cite{Popolizio2008}.
The formula is well defined. Indeed,
since $A$ is symmetric, the matrix $t_iA-\xi_jI$ is invertible if $\xi_j$ has nonzero imaginary part. In case of a real 
$\xi_\nu$, a direct computation shows that $\xi_\nu<0$ for $\nu\in\{1,\ldots,13\}$\footnote{The 
computation of $\xi_j$, $\theta_j$ can be carried out by using the polynomial coefficients listed
in \cite[Tab. III]{Cody1969} for $\nu=1,\ldots,14$. See also section~\ref{Implementation details}.}, $\nu$ odd,
so that $t_iA-\xi_\nu I$ is nonsingular as well. We refer the reader to 
section~\ref{Implementation details} for details on the 
computation of the weights and poles of the rational Chebyshev function \eqref{Cheb_ratfun}.
The number $\nu$ of terms in (\ref{Cheb_ratfun}) is closely related to the 
accuracy of the computed approximation. Indeed, it holds
(see, e.g., \cite{Carpenter1984})
$$ 
\sup_{\lambda\geq 0}|e^{-\lambda}-\mathcal{R}_\nu(\lambda)|\approx 10^{-\nu};
$$
a similar estimate holds for $\|e^{-A}-\mathcal{R}_\nu(A)\|_2$ for $A$ SPD.
Indeed, if $A=Q\Lambda Q^T$, $\Lambda=\mbox{diag}(\lambda_1,\ldots,\lambda_n),$
denotes the eigedecomposition of $A$, it holds
\begin{align*}
 \|e^{-A}-\mathcal{R}_\nu(A)\|_2&=\|e^{-\Lambda}-\mathcal{R}_\nu(\Lambda)\|_2=
%\|\mbox{diag}\left(e^{-\lambda_1}-\mathcal{R}_\nu(\lambda_1),\ldots,e^{-\lambda_n}-\mathcal{R}_\nu(\lambda_n)\right)\|_2\\
%& =
\max_{i=1,\ldots,n}|e^{-\lambda_i}-\mathcal{R}_\nu(\lambda_i)|.
\end{align*}
Few terms are thus needed to obtain a quite accurate approximation, for our purposes. 

The rational function approximation
%Therefore, the matrix exponentials in (\ref{gauss_lobatto}) can be well approximated by the rational Chebyshev function, 
%that is $e^{-t_iA}\approx \mathcal{R}_\nu\left(t_iA\right)$ for all $i=1,\ldots,\ell$. Such approximations require to 
(\ref{reducing_nu}) requires the computation of
several inverses of the form $(t_iA-\xi_jI)^{-1}$ for all $i=1,\ldots,\ell$, $j=1,\ldots,\nu$, which are, in general,
dense. This leads to the third approximation step above, that is
a banded approximation $\reallywidehat{(t_iA-\xi_jI)^{-1}}\approx (t_iA-\xi_jI)^{-1}$ with
%a banded approximations $S(t_i,\xi_j)\approx (t_iA-\xi_jI)^{-1}$ with
%$\beta_{S(t_i,\xi_j)}\ll n,$ is obtained. The quality of this approximation is ensured by the following
%%%%$\beta_{\reallywidehat{(t_iA-\xi_jI)^{-1}}}\ll n$. 
bandwidth much smaller than $n$.
The quality of this approximation is ensured by the following
result, which takes great advantage of the fact that the shifts $\xi_j$s are complex.
\begin{proposition}[\cite{Freund1989}]\label{Prop:Freund}
 Let $M=\upsilon_1 I+\upsilon_2 M_0$ be 
$\beta_M$-banded with $M_0$ Hermitian and $\upsilon_1,\upsilon_2\in\mathbb{C}$. 
 Define 
 $a:=(\lambda_{\max}(M)+\lambda_{\min}(M))/(\lambda_{\max}(M)-\lambda_{\min}(M))$ and
 $R:=\alpha+\sqrt{\alpha^2-1}$ with $\alpha=(|\lambda_{\max}(M)|+|\lambda_{\min}(M)|)/|\lambda_{\max}(M)-\lambda_{\min}(M)|$.
 Then,
\begin{equation}\label{estimate2}
\left|\left(M^{-1}\right)_{p,q}\right|\leq \frac{2R}{|\lambda_{\max}(M)-\lambda_{\min}(M)|}B(a)
\left(\frac{1}{R}\right)^{\frac{|p-q|}{\beta_M}},\quad p\neq q, 
\end{equation}
where, writing $a=\zeta_R\cos(\psi)+i\eta_R\sin(\psi)$,
$$ 
B(a):=\frac{R}{\eta_R\sqrt{\zeta_R^2-\cos^2(\psi)}(\zeta_R+\sqrt{\zeta_R^2-\cos^2(\psi)})},
$$
with $\zeta_R=(R+1/R)/2$ and $\eta_R=(R-1/R)/2$.
 \end{proposition}

If spectral estimates are available,
the entry decay of $(t_iA-\xi_jI)^{-1}$ can be cheaply predicted by means of \eqref{estimate2},
so that the sparsity pattern of the banded approximation 
$\reallywidehat{(t_iA-\xi_jI)^{-1}}$ to $(t_iA-\xi_jI)^{-1}$ 
can be estimated a-priori, during its computation. The actual procedure to determine
$\reallywidehat{(t_iA-\xi_jI)^{-1}}$ is discussed in section~\ref{Implementation details}.

The matrix exponential $e^{-t_iA}$ in \eqref{gauss_lobatto} is thus approximated by 
$$\reallywidehat{ \mathcal R}_\nu(t_iA):=
\sum_{j=1}^{\nu-1}2\mbox{Re}\left(\theta_j\reallywidehat{(t_iA-\xi_jI)^{-1}}\right)+
\theta_\nu \reallywidehat{(t_iA-\xi_\nu I)^{-1}}\approx \mathcal{R}_\nu\left(t_iA\right),
%\theta_\nu S(t_i,\xi_\nu)\approx \mathcal{R}_\nu\left(t_iA\right),
\quad i=1,\ldots,\ell.$$
We notice that the entries of the most external diagonals of 
$\mathcal{\reallywidehat R}_\nu(t_iA)$ might be small in magnitude.
To further reduce the bandwidth of $\mathcal{\reallywidehat R}_\nu(t_iA)$, we thus suggest to set to zero those components of 
$\mathcal{\reallywidehat R}_\nu(t_iA)$ that are smaller than $\epsilon_{quad}$, that
is, we replace the matrix $\mathcal{\reallywidehat R}_\nu(t_iA)$ with the matrix $\mathcal{\widetilde R}_\nu(t_iA)$ 
defined as follows
\begin{equation}\label{tildeR}
\mathcal{\widetilde R}_\nu(t_iA):=\mathcal{\reallywidehat R}_\nu(t_iA)-E_i, \quad \left(E_i\right)_{k,j}:=\left\{\begin{array}{l}
                                          \left(\mathcal{\reallywidehat R}_\nu(t_iA)\right)_{k,j}, \; \mbox{if } 
                                          \left|\left(\mathcal{\reallywidehat R}_\nu(t_iA)\right)_{k,j}\right|<\epsilon_{quad},\\
                                          0, \; \mbox{otherwise.}
                                         \end{array}\right.  
\end{equation}

Collecting all these observations, we have
\begin{equation}\label{X_B}
 X(\tau)\approx \frac{\tau}{2}\sum_{i=1}^\ell\omega_i 
 \mathcal{\widetilde R}_\nu(t_iA)D\mathcal{\widetilde R}_\nu(t_iA)=:X_B,
\end{equation}
and the bandwidth $\beta_{X_B}$ of $X_B$ is such that
$\beta_{X_B}\leq 2\max_i\{\beta_{\mathcal{\widetilde R}_\nu(t_iA)}\}+\beta_D$.
The overall procedure for computing $X_B$ is illustrated in Algorithm \ref{X_B_alg}.

\begin{algorithm}
\setcounter{AlgoLine}{0}
%\algsetup{linenosize=\small}
%\SetLine %% new algorithm2e: \SetAlgoLined
\caption{Numerical approximation of $X(\tau)$.\label{X_B_alg}}
\SetKwInOut{Input}{input}\SetKwInOut{Output}{output}
%%%%%%%%%%% INPUT %%%%%%%%%%%
\Input{$A\in\mathbb{R}^{n\times n},$ $A$ SPD, $D\in\mathbb{R}^{n\times n}$, 
$\nu\in\mathbb{N}$, $\epsilon_{B},\epsilon_{quad},\tau>0$}
%$\nu\in\mathbb{N}$, $\epsilon_{S(t_i,\xi_j)},\epsilon_{spy},\tau>0$}
%%%%%%%%%%% OUTPUT %%%%%%%%%%%
\Output{$X_B\in\mathbb{R}^{n\times n}$, $X_B\approx X(\tau)$}
%%%%%%%%%%%%%%%%%%%%%%%%%%%%%%%%%%%
\BlankLine
\nl Compute $t_i$, $\omega_i$, $i=1,\ldots,\ell$, for the Gauss-Lobatto formula \eqref{gauss_lobatto}\\
\nl Compute $\xi_j$, $\theta_j$, $j=1,\ldots,\nu$, for the rational Chebyshev approximation 
\eqref{Cheb_ratfun}\\
\nl Set $X_B=0$\\
\For{$i = 1,\dots,\ell$ }{
%\For{$j=1,\ldots,\nu$}{
\nl For $j=1, \ldots, \nu$ compute $\reallywidehat{(t_iA-\xi_jI)^{-1}}$ \\   % $S(t_i,\xi_j)$\\
%}
\nl Set $\mathcal{\reallywidehat R}_\nu(t_iA):=
\sum_{j=1}^{\nu-1}2\mbox{Re}\left(\theta_j \reallywidehat{(t_iA-\xi_jI)^{-1}}\right)+\theta_\nu 
\reallywidehat{(t_iA-\xi_\nu I)^{-1}}$\\
%\sum_{j=1}^{\nu-1}2\mbox{Re}\left(\theta_jS(t_i,\xi_j)\right)+\theta_\nu S(t_i,\xi_\nu)$\\
\nl Compute $\mathcal{\widetilde R}_\nu(t_iA)$ as in \eqref{tildeR}\\
\nl Set $X_B=X_B+\omega_i\mathcal{\widetilde R}_\nu(t_iA)D\mathcal{\widetilde R}_\nu(t_iA)$\\
 }
\nl Set $X_B=\frac{\tau}{2}X_B$\\
\end{algorithm}
%%%

%%%%%%%%%%%%%%%%%%%%%%%%%%%%%%%%%%%%%%%%%%%%%%%%%%%%%%%%%%%%%%%%%%%%%%%%%%%%%%%%%%%%%%%%%%%%%%%%%%%%%%%%%%%%%%%%%%%%%%%%
\subsection{Implementation details for computing $X_B$}\label{Implementation details}
In this section we illustrate some details to efficiently implement Algorithm~\ref{X_B_alg}.
%
% An adaptive Gauss-Lobatto quadrature formula is implemented in \eqref{gauss_lobatto} to preserve accuracy while
% avoiding chosing in advance the number of nodes $\ell$; this was performed
% by using a matrix-oriented version of the procedure in \cite[Section 4.5]{Gander2000}.

For given coefficients of the numerator and denominator polynomials (see, e.g., 
\cite{Cody1969}),
the weights and poles of the rational Chebyshev function \eqref{Cheb_ratfun} 
can be computed by the residue theorem, implemented in Matlab via the function {\tt residue}.
%In our experiments we always used $\nu=7$, thus providing a maximum attainable accuracy of the order 
%of $10^{-7}$.

%%%
%$$ [\mathtt{theta},\mathtt{xi}]=\mathtt{residue}(\mathtt{p},\mathtt{q}),$$
%%%
%where $\mathtt{theta}=(\theta_1,\ldots,\theta_\nu)$, $\mathtt{xi}=(\xi_1,\ldots,\xi_\nu)$ and
%the entries of the vectors 
%$\mathtt{p}=(\mathtt{p}_{\nu},\ldots,\mathtt{p}_0)$,
%$\mathtt{q}=(\mathtt{q}_{\nu},\ldots,\mathtt{q}_0)$ are listed in \cite{Cody1969} 
%for $\nu\in\{1,\ldots,14\}$. We always employ $\nu=7$.

%We now focus on the computation of $S(t_i,\xi_j)$ that approximates $(t_iA-\xi_jI)^{-1}$ since it consists of
%one of the most expensive step of the procedure.
The approximation of $(t_iA-\xi_jI)^{-1}$ for all considered $i$'s and $j$'s is the most time
consuming part of the process to obtain $X_B$. This is performed by using a sparse approximate inverse approach,
which has been extensively studied in the context of preconditioning techniques for solving 
large scale linear systems;
see, e.g., \cite{Benzi1999},\cite{Benzi2003},\cite{Bertaccini2004}.
Furthermore, many packages such as
SPAI\footnote{{\tt https://cccs.unibas.ch/lehre/software-packages/}} and
FSAIPACK\footnote{{\tt http://hdl.handle.net/11577/3132741}}
are available on-line for its computation.
Unfortunately, open software seldom handles complex arithmetic, as it occurs here whenever
the poles have nonzero imaginary part.
%we have to deal with complex arithmetic as
%$\xi_j\in\mathbb{C}$, $\Im(\xi_j)\ne 0$.

%Thanks to Proposition~\ref{Prop:Freund}, a rather sharp estimate of the entry-wise decay of $(t_iA-\xi_jI)^{-1}$ is available
%and this can be exploited in the computation of $S(t_i,\xi_j)$. In particular, the sparsity pattern of $S(t_i,\xi_j)$ can be easily
%estimated a-priori, so that sophisticated adaptive strategies for computing $S(t_i,\xi_j)$ are not necessary.

With the notation in Proposition \ref{Prop:Freund}, we have
$$
\left|\left((t_iA-\xi_jI)^{-1}\right)_{p,q}\right|\leq \frac{2R}{|\lambda_2-\lambda_1|}B(a)
\left(\frac{1}{R}\right)^{\frac{|p-q|}{\beta_A}},\quad p> 1,
$$ 
and this allows us to explicitly compute only those entries that are above a given tolerance, taking
symmetry into account.
%and the matrix $(t_iA-\xi_jI)^{-1}$ is not fully computed but only its most relevant entries are calculated.
%In particular, exploiting the symmetry of $A$, we can explicitly compute only the significant components
%in the lower triangular part of $(t_iA-\xi_jI)^{-1}$.
%%
%{\color{red} INCONSISTENT USE OF LETTERS/EXPRESSIONS

For every column $q=1,\ldots,n$, 
 we compute $\bar p_q\,(t_i,\xi_j)$ such that
%and, fixing a threshold $\epsilon_{S(t_i,\xi_j)}$, we compute $\bar p\,(t_i,\xi_j)$ such that
%%
\begin{equation}\label{bar_p}
\bar p_q\,(t_i,\xi_j)=\mbox{argmin}\left\{p>1, \mbox{ s.t. } \frac{2R}{|\lambda_2-\lambda_1|}B(a)
\left(\frac{1}{R}\right)^{\frac{|p-q|}{\beta_A}}\leq\epsilon_{B}\right\}, 
\end{equation}
%\left(\frac{1}{R}\right)^{\frac{|p-1|}{\beta_A}}<\epsilon_{S(t_i,\xi_j)}\right\}.$$
%%
where $\epsilon_{B}$ is a given threshold.
Defining $\widehat p_q\,(t_i,\xi_j):= \min\{n,q+\bar p_q\,(t_i,\xi_j)\}$, we calculate  
$\left((t_iA-\xi_jI)^{-1}\right)_{p,q},$ $q=1,\ldots,n,$ $p=q,\ldots,\widehat p_q(t_i,\xi_j)$
that are the most meaningful entries of $t_iA-\xi_jI$. Indeed, only for these indices, it holds
$|\left((t_iA-\xi_jI)^{-1}\right)_{p,q}|\geq\epsilon_B$.
% where
% %%
% \begin{equation}\label{index_SPAI}
%  p_q'(t_i,\xi_j)=\left\{\begin{array}{l}
%         1, \mbox{ if } q-\bar p\,(t_i,\xi_j)< 1,\\
%         q-\bar p\,(t_i,\xi_j), \mbox{ otherwise},
%        \end{array}\right.\;
% p_q''(t_i,\xi_j)=\left\{\begin{array}{l}
%         n, \mbox{ if } q+\bar p\,(t_i,\xi_j)> n,\\
%         q+\bar p\,(t_i,\xi_j), \mbox{ otherwise}.
%        \end{array}\right.
% \end{equation}
% %%
%  In particular, once $\bar p\,(t_i,\xi_j)$ is computed,
To this end, we perform a complex (symmetric)
LDLt factorization of  
$t_iA-\xi_jI$, that is $t_iA-\xi_jI=L(t_i,\xi_j)D(t_i,\xi_j)L(t_i,\xi_j)^T$, and solve
%}
%%
\begin{equation}\label{LU}
L(t_i,\xi_j)D(t_i,\xi_j)L(t_i,\xi_j)^Ts_q=e_q,\quad q=1,\ldots,n. 
\end{equation}
We do not compute all entries of $s_q$ but only those in 
position $r$, $r=q,\ldots,\widehat p_q\,(t_i,\xi_j)$,
suitably \! performing \!
the \!forward \!and \!backward substitution \!with $\!L(t_i,\xi_j)$ and $\!L(t_i,\xi_j)^T$ respectively.
%$\mathcal{O}(2\beta_A(p_q''(t_i,\xi_j)-p_q'(t_i,\xi_j)+1))\,flops$. Exploiting the symmetry of $A$, 
%we can compute only $p_q''(t_i,\xi_j)-q+1$
%entries, from the $q$-th to the $p_q''(t_i,\xi_j)$-th, decreasing to 
%$\mathcal{O}(\beta_A(p_q''(t_i,\xi_j)-p_q'(t_i,\xi_j)+1)(p_q''(t_i,\xi_j)-q+1))$ the number of operations
%required to compute the most meaningful entries
%of $s_q$. 
The computed $s_q$ approximates the $q$-th column of $(t_iA+\xi_jI)^{-1}$, in particular, 
$(s_q)_r=((t_iA+\xi_jI)^{-1}e_q)_r$ for $r=q,\ldots,\widehat p_q\,(t_i,\xi_j)$.

If $\mathfrak{S}=[s_1,\ldots,s_n]$ and $\mathfrak{s}$ denotes its diagonal,
we define $\reallywidehat{(t_iA+\xi_jI)^{-1}}:=\mathfrak{S}+\mathfrak{S}^T-\mbox{diag}(\mathfrak{s})$,
and it holds
$\|\reallywidehat{(t_iA+\xi_jI)^{-1}}-(t_iA+\xi_jI)^{-1}\|_{\max}<\epsilon_{B}$. 
 Moreover, $\reallywidehat{(t_iA+\xi_jI)^{-1}}$ is a banded matrix with bandwidth
$$
\beta_{\reallywidehat{(t_iA+\xi_jI)^{-1}}}=\max_{q=1,\ldots,n}\widehat p_q\,(t_i,\xi_j).
$$
Therefore,
the bandwidth of the final approximation $X_B$ in \eqref{X_B} will be such that 
$\beta_{X_B}\leq 2\max_{i,j}\beta_{\reallywidehat{(t_iA+\xi_jI)^{-1}}}+\beta_D$.

 The overall procedure is summarized in Algorithm~\ref{sparseInv} where complex arithmetic is 
necessary due to the presence of the shift $\xi_j$.
%
%
%, and 
%$\|s_q-(t_iA+\xi_jI)^{-1}e_q\|_{max}<\epsilon_{B}$.
%$\|s_q-(t_iA+\xi_jI)^{-1}e_q\|_{max}<\epsilon_{S(t_i,\xi_j)}$.
%
The computational cost of the complete algorithm is proportional to the problem size $n$. Indeed,
since $t_iA+\xi_jI$ is a $\beta_A$-banded matrix, the computation of $L(t_i,\xi_j)$ and $D(t_i,\xi_j)$ requires 
%Since $A$, and thus $t_iA+\xi_jI$, is a $\beta_A$-banded matrix, the computation of $L(t_i,\xi_j)$ and $D(t_i,\xi_j)$ requires 
$\mathcal{O}(n\beta_A)$ flops.
 Notice that the computational core of Algorithm~\ref{sparseInv} consists of inner products with vectors of 
length (at most) $\widehat p_q\,(t_i,\xi_j)-q+1$.
  Therefore, the computation of the $\widehat p_q\,(t_i,\xi_j)-q+1$ entries of $s_q$ costs
$\mathcal{O}(\widehat p_q\,(t_i,\xi_j)-q)$ flops.
The overall computational cost of \eqref{LU}, for all $q$, thus amounts to 
$\mathcal{O}(n\max_q\{\widehat p_q\,(t_i,\xi_j)-q\})$ flops.

The matrix $\reallywidehat{(t_iA+\xi_jI)^{-1}}$ has to be computed for all 
%$\|S(t_i,\xi_j)-(t_iA+\xi_jI)^{-1}\|_{max}<\epsilon_{S(t_i,\xi_j)}$. The matrix $S(t_i,\xi_j)$ has to be computed for all 
$i=1,\ldots,\ell$, $j=1,\ldots,\nu$, leading to a computational cost of 
$\mathcal{O}(n\ell\nu\max_{q,i,j}\{\widehat p_q\,(t_i,\xi_j)-q\})$ flops.
Moreover,
thanks to the observation in \eqref{reducing_nu}, we 
can compute $\reallywidehat{(t_iA+\xi_jI)^{-1}}$, for $i=1,\ldots,\ell$, and only few terms in $j$.
Fixing $i\in\{1,\ldots,\ell\}$, the matrices $\reallywidehat{(t_iA+\xi_jI)^{-1}}$, $j=1,\ldots,\nu$, $j$ odd,
are computed in parallel, thus
 decreasing the cost of the overall procedure to
$\mathcal{O}(n\ell\max_{q,i,j}\{\widehat p_q\,(t_i,\xi_j)-q\})$ flops.

Optimal parameter $\nu$ and thresholds $\epsilon_B$, $\epsilon_{quad}$ requested by Algorithm~\ref{X_B_alg}
may be tricky to determine in an automatic manner. Our numerical experience  seems to
suggest that by setting $\epsilon_B=\epsilon_{quad}$ and $\nu=\lfloor\log(1/\epsilon_{quad})\rfloor-1$, the
performance is not affected, while we are able to save the user from selecting two more parameters.
With these choices we observed that
$\|e^{-t_iA}-\mathcal{\widetilde R}_\nu(t_iA)\|_2\approx \epsilon_{quad}$ and this accuracy is maintained also by 
the adaptive quadrature formula.

%%%%%%%%%%%%%%%%%%%%%%%%%%%%%%%%%%%%%%%%%%%%%%%%%%%%%%%%%%%%%%%%%%%%%%%%%%%%%%%%%%%%%%%%%%%%%%%%%%%%%%%%%%%%%%%%%%%%%%%%
\subsection{Approximating $e^{-\tau A} X e^{-\tau A}$ by a low-rank matrix}\label{IItau}
%\subsection{Approximating $\mathcal{II}_\tau$ by a low-rank matrix}\label{IItau}
%
We next turn our attention to the second component in (\ref{eqn:splitting}),
 $e^{-\tau A} X e^{-\tau A}$.
%%
%For $\tau=0$, $e^{-\tau\Lambda}=I$ and \eqref{Xtau} corresponds to \eqref{Xintegral}.
%the characterization of the solution $X$ 
%to \eqref{main.eq} in terms of the eigendecomposition of $A$. See. e.g., \cite{Palitta2016a}.
We show that for large $\tau$ this matrix can be well 
approximated by a low-rank matrix. % for large $\tau$.
 In the following we shall assume that the eigenvalues of the SPD matrix $A$
%$\lambda_1\geq\ldots\geq\lambda_n>0$ 
decay more than linearly, so as to ensure the 
low numerical rank of $e^{-\tau A}$ for $\tau$ sufficiently large.

\begin{proposition}\label{Cor.IItau}
Let $\lambda_1\geq\ldots\geq\lambda_n>0$ be the eigenvalues of $A$ and $X$ as in \eqref{Xintegral}. 
% Let $\mathcal{II}_\tau$ be defined as in (\ref{Xtau}). 
Then, $\mbox{rank}(e^{-\tau A} X e^{-\tau A})\searrow0$ as
 $\tau\rightarrow+\infty$, and
 there exists a matrix $X_L\in\mathbb{R}^{n\times n}$, $\mbox{rank}(X_L)=\bar \ell\ll n$, such that
 %%
%{\color{red}
 \begin{equation}\label{EstimateXL}
\|e^{-\tau A} X e^{-\tau A}-X_L\|_2^2\leq \frac{3}{4\lambda_n^2}e^{-2\tau(\lambda_n+\lambda_{n-\bar \ell})}\|D\|_F^2,  
 \end{equation}
%}
%%

\end{proposition}
%%
%\begin{proof}

{\it Proof.}
Let $A=Q\Lambda Q^T$, $\Lambda=\mbox{diag}\,(\lambda_1,\ldots,\lambda_n)$
%$\lambda_1\geq\ldots\geq\lambda_n>0$, 
be the eigendecomposition of $A$. Then, we can write
$e^{-\tau A} X e^{-\tau A}= Qe^{-\tau\Lambda}(Q^TXQ)e^{-\tau\Lambda}Q^T= Qe^{-\tau\Lambda}Ye^{-\tau\Lambda}Q^T$, 
where $Y\in\mathbb{R}^{n\times n}$ is such that
$\Lambda Y+Y\Lambda=Q^TDQ$.
We notice that $e^{-\tau\lambda_i}\leq e^{-\tau\lambda_j}$ for all $j\leq i$ and 
$e^{-\tau\lambda_i}\rightarrow0,$ $\tau\rightarrow+\infty$, for all $i=1,\ldots,n$.
 Hence, $e^{-\tau A} X e^{-\tau A}=Qe^{-\tau\Lambda}Ye^{-\tau\Lambda}Q^T$
 is numerically low-rank as $\tau\rightarrow+\infty$ since 
 $\mbox{rank}(e^{-\tau\Lambda})=\mbox{rank}\left(\mbox{diag}(e^{-\tau\lambda_1},\ldots,e^{-\tau\lambda_n})\right)
 \searrow0$ as $\tau\rightarrow+\infty$.

 For a fixed $\bar\ell$, we consider the partition  $Q=[Q_1,Q_2],\; Q_1\in\mathbb{R}^{n\times (n-\bar\ell)},
 Q_2\in\mathbb{R}^{n\times \bar\ell}$, $e^{-\tau\Lambda}=\mbox{blkdiag}(e^{-\tau\Lambda_1},e^{-\tau\Lambda_2})$, 
 $\Lambda_1=\mbox{diag}(\lambda_1,\ldots,\lambda_{n-\bar\ell}),
 \Lambda_2=\mbox{diag}(\lambda_{n-\bar\ell+1},\ldots,\lambda_n)$, and 
 $  Y=[Y_{11}, Y_{12}; Y_{21}, Y_{22}]$
with blocks $Y_{st}$, $s,t=1,2$, of conforming dimensions, that is
$Y_{st}$ is the solution of the Sylvester equation $\Lambda_sY_{st}+Y_{st}\Lambda_t=Q_s^T DQ_t$, $s,t=1,2$.
  Then,
 {\small
 \begin{align*}
  e^{-\tau A} X e^{-\tau A}&= Qe^{-\tau\Lambda}Ye^{-\tau\Lambda}Q^T=
  [Q_1,Q_2]\begin{bmatrix}
                                         e^{-\tau\Lambda_1} & \\
                                             & e^{-\tau\Lambda_2} \\
                                        \end{bmatrix}\begin{bmatrix}
                                         Y_{11} & Y_{12}\\
                                           Y_{21}  & Y_{22} \\
                                        \end{bmatrix}\begin{bmatrix}
                                         e^{-\tau\Lambda_1} & \\
                                             & e^{-\tau\Lambda_2} \\
                                        \end{bmatrix}\begin{bmatrix}
                                         Q_1^T\\
                                           Q_2^T \\
                                        \end{bmatrix}. \\
 % &=Q_1e^{-\tau\Lambda_1}Y_{11}e^{-\tau\Lambda_1}Q_1^T+Q_2e^{-\tau\Lambda_2}Y_{21}e^{-\tau\Lambda_1}1Q_1^T
 % +Q_1e^{-\tau\Lambda_1}Y_{12}e^{-\tau\Lambda_2}Q_2^T+Q_2e^{-\tau\Lambda_2}Y_{22}e^{-\tau\Lambda_2}Q_2^T.                                      
\end{align*}
 }
 Defining $X_L:=Q_2e^{-\tau\Lambda_2}Y_{22}e^{-\tau\Lambda_2}Q_2^T$, $\mbox{rank}(X_L)=\bar\ell$, we have
\begin{eqnarray*}
  \|e^{-\tau A} X e^{-\tau A}-X_L\|_2^2&=&\left\|[Q_1,Q_2]\begin{bmatrix}
                                         e^{-\tau\Lambda_1} & \\
                                             & e^{-\tau\Lambda_2} \\
                                        \end{bmatrix}\begin{bmatrix}
                                         Y_{11} & Y_{12}\\
                                           Y_{21}  & 0 \\
                                        \end{bmatrix}\begin{bmatrix}
                                         e^{-\tau\Lambda_1} & \\
                                             & e^{-\tau\Lambda_2} \\
                                        \end{bmatrix}\begin{bmatrix}
                                         Q_1^T\\
                                           Q_2^T \\
                                        \end{bmatrix}\right\|_2^2\\ 
                                        &=&\left\|\begin{bmatrix}
                                         e^{-\tau\Lambda_1} & \\
                                             & e^{-\tau\Lambda_2} \\
                                        \end{bmatrix}\begin{bmatrix}
                                         Y_{11} & Y_{12}\\
                                           Y_{21}  & 0 \\
                                        \end{bmatrix}\begin{bmatrix}
                                         e^{-\tau\Lambda_1} & \\
                                             & e^{-\tau\Lambda_2} \\
                                        \end{bmatrix}\right\|_2^2\\
  %\|Q_1e^{-\tau\Lambda_1}Y_{11}e^{-\tau\Lambda_1}Q_1^T
  %+Q_2e^{-\tau\Lambda_2}Y_{21}e^{-\tau\Lambda_1}Q_1^T+Q_1e^{-\tau\Lambda_1}Y_{12}e^{-\tau\Lambda_2}Q_2^T\|_2^2\\
  &\leq &\left(\|e^{-\tau\Lambda_1}Y_{11}e^{-\tau\Lambda_1}\|_2+
  \|e^{-\tau\Lambda_2}Y_{21}e^{-\tau\Lambda_1}\|_2+\|e^{-\tau\Lambda_1}Y_{12}e^{-\tau\Lambda_2}\|_2\right)^2\\
  &\leq &\left( e^{-2\tau\lambda_{n-\bar\ell}}\|Y_{11}\|_2+e^{-\tau(\lambda_n+\lambda_{n-\bar\ell})}\|Y_{21}\|_2
 +e^{-\tau(\lambda_n+\lambda_{n-\bar\ell})}\|Y_{12}\|_2\right)^2\\
 &\leq & \left(e^{-2\tau\lambda_{n-\bar\ell}}\|Y_{11}\|_F+e^{-\tau(\lambda_n+\lambda_{n-\bar\ell})}\|Y_{21}\|_F
 +e^{-\tau(\lambda_n+\lambda_{n-\bar\ell})}\|Y_{12}\|_F\right)^2\\
 &\leq & \left(e^{-2\tau\lambda_{n-\bar\ell}}+2e^{-\tau(\lambda_n+\lambda_{n-\bar\ell})}\right)^2\|Y\|_F^2 \\
 &\leq & \left(e^{-\tau\lambda_{n-\bar\ell}}+2e^{-\tau\lambda_n}\right)^2e^{-2\tau\lambda_{n-\bar\ell}}\|Y\|_F^2 
 \leq 3e^{-2\tau(\lambda_n+\lambda_{n-\bar\ell})}\|Y\|_F^2.
 \end{eqnarray*}
 Since $Y$ is such that $\Lambda Y+Y\Lambda=Q^T DQ$, it holds 
 $\|Y\|_F^2\leq \frac{\|D\|_F^2}{4\lambda_n^2}.$ Therefore,
 we can write
 %%
 %\begin{align*}
$$
 \|e^{-\tau A} X e^{-\tau A}-X_L\|_2^2
\leq 
\frac{3}{4\lambda_n^2} e^{-2\tau(\lambda_n+\lambda_{n-\bar \ell})}\|D\|_F^2 . \qquad \square
$$
 %\end{align*}
 %%
%  Moreover, recalling that $Y$ is the solution of the Lyapunov 
% equation $\Lambda Y+Y\Lambda=Q^T DQ$, we have 
% $\|Y\|_2\leq \frac{\|Q^T DQ\|_2}{2\lambda_n}=\frac{\| D\|_2}{2\lambda_n}$, so that
% %%
%  \begin{align*} 
%  \|\mathcal{II}_\tau-X_L\|_2^2&\leq e^{-2\tau\lambda_{n-\bar\ell}} \left(e^{-2\tau\lambda_{n-\bar\ell}}+2e^{-2\tau\lambda_n}\right)
%  \frac{\|D\|_2^2}{4\lambda_n^2}\\
% &\leq e^{-2\tau\lambda_{n-\bar\ell}} \left(e^{-2\tau\lambda_n}+2e^{-2\tau\lambda_n}\right)\frac{\|D\|_2^2}{4\lambda_n^2}\\
% &=\frac{3}{4}e^{-2\tau(\lambda_n+\lambda_{n-\bar\ell})}\frac{\|D\|_2^2}{\lambda_n^2}.
%  \end{align*}
 %%
%\end{proof}
%%

% {\color{red}
% ----
% 
% MOVE? Let $Y_{22}=\widehat Y_{22} \widehat Y_{22}^T$ be a Cholesky factorization of $Y_{22}$.
% Then Proposition~\ref{Cor.IItau} shows that the second term defining $X$ can be approximated as
% $\mathcal{II}_\tau\approx X_L=\widehat X_L\widehat X_L^T$ where $\widehat X_L=Q_2e^{-\tau\Lambda_2}\widehat Y_{22}$
% belongs to the space spanned by the eigenvectors associated with the smallest $\bar \ell$ eigenvalues of $A$.
% 
% I WOULD REMOVE IT.
% -----
% }

The proof is constructive, since it provides an explicit form for $X_L$, that is 
 $X_L=Q_2e^{-\tau\Lambda_2}Y_{22}e^{-\tau\Lambda_2}Q_2^T$, where $\Lambda_2$ contains the $\bar\ell$ 
eigenvalues closest to
the origin, and the columns of $Q_2$ constitute
 the associated invariant subspace basis; $Y_{22}$ is
the solution of a reduced Lyapunov equation.

Depending on the eigenvalue distribution,
Proposition~\ref{Cor.IItau} shows that a good approximation may be obtained
by using only few of the eigenvectors of $A$, where however
$\bar \ell$ is not known a priori. 
Moreover, the computation of $\bar \ell$ eigenpairs of a large matrix, though SPD and banded, may be too expensive.
We thus propose to employ a Krylov subspace type procedure to capture information
on the relevant portion of the eigendecomposition of $A$. More precisely, let
$\mathbf{K}_m(A^{-1},v):=\mbox{Range}([v,A^{-1}v,\ldots,A^{-m+1}v])$ 
where $v\in\mathbb{R}^n$ is a random vector with unit norm, let the columns of
$V_m=[v_1,\ldots,v_m]\in\mathbb{R}^{n\times m}$, $m\ll n$, be an orthonormal basis of $\mathbf{K}_m(A^{-1},v)$
and $K_m=V_m^TAV_m$.
If $V_m$ is such that $e^{-\tau A} \approx V_m e^{-\tau K_m} V_m^T$, then we approximate
\begin{equation}\label{eqn:expA}
%\begin{array}{rll}
%  e^{-\tau A} X e^{-\tau A}&\approx&  V_m\left(V_m^Te^{-\tau A}Xe^{-\tau A} V_m\right)V_m^T\\
% &\approx&  V_m\left(V_m^Te^{-\tau A}V_m\right)\left(V_m^TXV_m\right)\left(V_m^Te^{-\tau A} V_m\right)V_m^T\\
%& =&V_m\left(e^{-\tau K_m} \left(V_m^TXV_m\right) e^{-\tau K_m}\right)V_m^T, 
%\end{array}
e^{-\tau A} X e^{-\tau A} \approx V_m\left(e^{-\tau K_m} \left(V_m^TXV_m\right) e^{-\tau K_m}\right)V_m^T .
 \end{equation}
%%
% where $K_m=V_m^TAV_m$. 
The use of $A^{-1}$ in the definition of the Krylov subspace
$\mathbf{K}_m(A^{-1},v)$ is geared towards a fast approximation of the smallest eigenvalues of $A$
and the associated eigenvectors, particularly suitable for the approximation of the exponential
\cite{Eshof2006}.
 Since $e^{-\tau A}$ and $A$ commute,
we observe that $e^{-\tau A} X e^{-\tau A}$ solves the Lyapunov equation
$$
A e^{-\tau A} X e^{-\tau A}  + e^{-\tau A} X e^{-\tau A}  A = e^{-\tau A} D e^{-\tau A} .
$$
Substituting the approximation in (\ref{eqn:expA}) we can define the following residual matrix
\begin{eqnarray*}
{\cal R}_m &=& A V_m e^{-\tau K_m} ( V_m^T XV_m ) e^{-\tau K_m} V_m^T +
V_m e^{-\tau K_m} ( V_m^T XV_m ) e^{-\tau K_m} V_m^T A \\
& & \,\, - V_m e^{-\tau K_m} ( V_m^T D V_m) e^{-\tau K_m} V_m^T.
\end{eqnarray*}
To complete the approximation, we need to replace $V_m^T X V_m$ with 
some easily computable quantity $Z_m \approx V_m^T X V_m$, so that
the final approximation will be
$$
e^{-\tau A} X e^{-\tau A} \approx V_m\left(e^{-\tau K_m} Z_m e^{-\tau K_m}\right)V_m^T.
$$
%If $m=n$, then $X=V_m(V_m^T X V_m)V_m^T$. For $m<n$,
%we can consider the approximation $X\approx V_m Z_m V_m^T$ where $Z_m$ tends to approximate
%$V_m^TXV_m$ in some way. A natural choice consists of
%mimicking the projection strategy employed in the solution of Lyapunov equations with low-rank right-hand side,
%though our $D$ is not necessarily low rank. In fact,
%our aim here is not to obtain $Z_m$ so as to accurately solve the original Lyapunov equation.
%Instead, since only the quantity $V_m^T X V_m$ is of interest, it suffices to approximate well the portion
%of $X$ belonging to the range of $V_m$.
%For this reason, 
%We project the original equation onto the space spanned by Range($V_m$) by imposing the
To this end, we impose the
standard matrix Galerkin condition on the residual matrix $\mathcal{R}_m$, that is 
%:=AV_mZ_mV_m^T+V_mZ_mV_m^TA-D$, that is
$V_m^T\mathcal{R}_mV_m=0$. Explicitly writing all terms in this matrix equation 
 leads to the solution of the following $m\times m$ Lyapunov equation
 \begin{equation}\label{eq:projected}
  K_mZ_m+Z_mK_m=D_m,
 \end{equation}
where $D_m=V_m^TDV_m$;  see, e.g., \cite{Simoncini2016}. Note that the matrix exponential terms 
$e^{-\tau K_m}$ simplify.
For $m\ll n$ equation \eqref{eq:projected} could be solved
by decomposition-based methods such as the Bartels-Stewart 
method \cite{Bartels1972}, or its symmetric version, the Hammarling method \cite{Hammarling1982}.
%since $D$, so that $P_m$, is symmetric. 
We opt for the explicit computation, since the eigendecomposition is also used
to get the final matrix $S_m$.
Let $K_m=\Pi_m\Psi_m\Pi_m^T$ with $\Psi_m=\mbox{diag}(\psi_1,\ldots,\psi_m)$
be the eigendecomposition of $K_m$. Plugging these matrices in 
\eqref{eq:projected} gives
\begin{equation}\label{eq:projected2}
\Psi_m\widehat Z_m+\widehat Z_m\Psi_m=\Pi_m^TD_m\Pi_m, 
\end{equation}
where $\widehat Z_m=\Pi_m^TZ_m\Pi_m$. Since $\Psi_m$ is diagonal, we can write 
$\left(\widehat Z_m\right)_{i,j}=\frac{\left(\Pi_m^TD_m\Pi_m\right)_{i,j}}{\psi_i+\psi_j}$.
With $\widehat Z_m$ at hand, and with its eigendecomposition being
$\widehat Z_m=W\Theta W^T$, we can set
\begin{eqnarray}\label{eqn:S}
S_m:=V_m\left(\Pi_me^{-\tau \Psi_m}W \Theta^{1/2}\right), \qquad {\rm so \,\, that} \quad
e^{-\tau A} X e^{-\tau A}\approx S_m S_m^T. 
\end{eqnarray}

A rank reduction of $S_m$ can be performed if some of the diagonal elements of $\Theta^{1/2}$
fall below a certain tolerance, so that the corresponding columns can be dropped. This post-processing
gives rise to a thinner matrix $S_m$, with fewer than $m$ columns.

% (with eigenvalues $\theta_i$'s ordered non-increasingly)
%Once \eqref{eq:projected2} is solved, a low-rank truncation of $\widehat Z_m$
% is performed. In particular, the eigedecomposition 
%$\widehat Z_m=W\Theta W^T$ (with eigenvalues $\theta_i$'s ordered non-increasingly)
%can be computed dropping those eigenvalues which are below a certain tolerance, that is $\Theta=\mbox{blkdiag}(\Theta_1,
% \Theta_2)$, $W=[W_1,W_2]$ with $\|\Theta_2\|_F\leq \epsilon_{trunc}$. 
% Then, we set $ \Upsilon=W_1\sqrt{\Theta_1}\in\mathbb{R}^{m\times s}$, $s\leq m$; in this way 
% $\|\widehat Z_m-\Upsilon\Upsilon^T\|_F\leq\epsilon_{trunc}$.
%Then $\mathcal{II}_\tau\approx X_mX_m^T$, 
%$X_m:=V_m\left(\Pi_me^{-\tau \Psi_m}\Upsilon\right)\in\mathbb{R}^{n\times s}$.
%and $X''_m:=V_m\left(\Pi e^{-\tau \Psi_m}\Upsilon_2\right)$.

Assume that the matrix $X_B$ in \eqref{X_B} has been already computed.
Then the space $\mathbf{K}_m(A^{-1},v)$ is expanded until the residual norm of the original problem
$$
\|R\|_F:=\|A(X_B+S_mS_m^T)+(X_B+S_mS_m^T)A-D\|_F,
$$
is sufficiently small.
Exploiting the sparsity of $X_B$ and the low-rank property of $S_mS_m^T$, 
the quantity $\|R\|_F$ can be computed in $\mathcal{O}(sn)$ flops, where $s=\mbox{rank}(S_m)$, without 
the construction of the large and dense matrix $R$. 
See section~\ref{Implementation details2} for more details. 
The overall procedure is summarized in Algorithm~\ref{iterativeZ}.

\begin{algorithm}
%\algsetup{linenosize=\small}
\setcounter{AlgoLine}{0}
%\SetLine %% new algorithm2e: \SetAlgoLined
\caption{Iterative approximation of $e^{-\tau A} X e^{-\tau A}$.\label{iterativeZ}}
\SetKwInOut{Input}{input}\SetKwInOut{Output}{output}
%%%%%%%%%%% INPUT %%%%%%%%%%%
\Input{$A\in\mathbb{R}^{n\times n},$ $A$ SPD., $D,X_B\in\mathbb{R}^{n\times n}$, $v\in\mathbb{R}^{n}$, 
$\tau,\epsilon_{res},\epsilon_{it}>0$, $m_{\max}\in\mathbb{N}$}
%%%%%%%%%%% OUTPUT %%%%%%%%%%%
\Output{$S_m,\in\mathbb{R}^{n\times s}$, $s\ll n$, such that $S_m S_m^T\approx e^{-\tau A} X e^{-\tau A}$}
%%%%%%%%%%%%%%%%%%%%%%%%%%%%%%%%%%%
\BlankLine
\nl Set $\mu=\|D\|_F$\\
%\nl Compute the Cholesky factorization $A=LL^T$ \\
\nl Set $V_1=v/\|v\|$  \\
\For{$m = 1,2,\dots$ until convergence }{
\nl Expand $K_m=V_m^TAV_m,$ $D_m=V_m^TDV_m$ \label{update}\\
\nl Compute the eigendecomposition $K_m=\Pi_m\Psi_m\Pi_m^T$	\\
\nl Solve $\Psi_m\widehat Z_m+\widehat Z_m\Psi_m=\Pi_m^TD_m\Pi_m$ \\
\nl Compute the eigendecomposition $\widehat Z_m=W\Theta W^T$ \\ %and retain its leading terms\\
\nl Set %$X_m:=V_m\left(\Pi_me^{-\tau \Psi_m}\Upsilon\right)$\\
$S_m:=V_m\left(\Pi_me^{-\tau \Psi_m}W \Theta^{1/2}\right)$ and reduce columns if desired \\
\nl Compute $\|R\|_F/\|D\|_F$ \label{res}\\
%:=\|A(X_B+X_mX_m^T)+(X_B+X_mX_m^T)A-D\|_F/\|D\|_F$ \label{res} \\
\nl \If{$\|R\|_F/\|D\|_F<\epsilon_{res}$ {\rm \textbf{or}} $|\|R\|_F-\mu|/\|R\|_F<\epsilon_{it}$ {\rm \textbf{or}} 
$m> m_{\max}$}{ 
\nl \textbf{Stop} \label{stop}\\ }
\nl $\widehat v=A^{-1}v_m$ \label{lin.solves}\\
\nl $\widetilde v$ $\leftarrow$ Orthogonalize $\widehat v$ w.r.t. $V_{m}$ \label{orth}\\
\nl Set $v_{m+1}=\widetilde v/\|\widetilde v\|$ and $V_{m+1}=[V_m,v_{m+1}]$\label{endArnoldi}\\
\nl Set $\mu=\|R\|_F$\\
}
\end{algorithm}
%%%

The two step procedure for the approximation of $X$ provides a threshold for
the final attainable accuracy, and in particular for $\|R\|_F$. Indeed,
assume that $X_B \ne X(\tau)$. Then the final residual cannot
go below the discrepancy $X(\tau) - X_B$ even if the low rank portion of the
solution is more accurate. Indeed,
\begin{eqnarray*}
R &=& A (X_B+S_mS_m^T) + (X_B+S_mS_m^T) A - D \\
&=& \undergroup{A(X_B-X(\tau))+ (X_B-X(\tau)) A} + 
\underbrace{A(X(\tau)+S_mS_m^T) + (X(\tau)+S_mS_m^T) A - D}_{R_{ideal}}.
%\undergroup{A(X(\tau)+S_mS_m^T) + (X(\tau)+S_mS_m^T) A - D}.
\end{eqnarray*}
%The quantity $R_{ideal}=A(X(\tau)+S_mS_m^T) + (X(\tau)+S_mS_m^T) A - D$ is the ideal (non-computable) residual
The matrix $R_{ideal}$   % =A(X(\tau)+S_mS_m^T) + (X(\tau)+S_mS_m^T) A - D$ 
is the ideal (non-computable) residual
one would obtain if the banded part were computed exactly, and we obtain
% $R$
%differs from this ideal residual by the quantity $A(X_B-X(\tau))+ (X_B-X(\tau)) A$, so that
%particular,
$$
\|R - R_{ideal} \|_F = \| A(X_B-X(\tau))+ (X_B-X(\tau)) A\|_F  \le 2 \, \|A\|_F \, \|X_B-X(\tau)\|_F  .
$$
%and
%$$
%|\|R\|_F - \|R_{ideal}\|_F | \le 2\, \|A\|_F \, \|X_B-X(\tau)\|_F.
%$$
%In other words, we expect that the final attainable residual norm $\|R\|_F$ is hindered by the
%inaccuracy of $X_B$.
 Therefore, even if $S_mS_m^T$ is accurate, $\|R\|_F$ may stagnate at the level
of $\|X_B-X(\tau)\|_F$. To limit this stagnation effect, we include a stopping criterion
that avoids iterating when the residual stops decreasing significantly, and in all our numerical experiments we set 
$\epsilon_{it}=\epsilon_{quad}$, where $\epsilon_{quad}$ is related to the
accuracy of $X_B$.

\subsection{Implementation details for computing the low rank part of the solution}\label{Implementation details2}
%\subsection{Implementation details for computing $S_m$}\label{Implementation details2}
We first notice that the update of the matrices $K_m=V_m^TAV_m,$ $D_m=V_m^TDV_m$ in line \ref{update} 
of Algorithm~\ref{iterativeZ} only requires the addition of one extra column and row at each
iteration. 
%few matrix-vector products at each iteration. Indeed,
%%
%$$ K_m=V_m^TAV_m=\begin{bmatrix}
%                 V_{m-1}^TAV_{m-1} & V_{m-1}^TAv_m \\
%                 v_m^TAV_{m-1} & v_m^TAv_m \\
%                 \end{bmatrix},
%$$
%so that, at the $m$th iteration, we need to compute only the vector $V_{m}^TAv_m\in\mathbb{R}^m$. Similarly for $P_m$.
%
Moreover, for the sake of robustness we perform a full basis orthogonalization at step \ref{orth},
though in exact arithmetic this would be ensured by the symmetry of $A$. Alternative computationally convenient
strategies would include a selective orthogonalization \cite{Parlett1979}.
Moreover, the linear systems with $A$ in line~\ref{lin.solves} can be solved by, e.g., a sparse Cholesky method.
%In spite of the symmetry of $A$, we prefer to perform a full orthogonalization in step \ref{orth}, altough different
%strategies as a Lanczos method with selective orthogonalization may be pursued as well. See, e.g. \cite{Parlett1979}.

The computational core of Algorithm~\ref{iterativeZ} is the residual norm calculation in line~\ref{res}.
%Clearly, we cannot compute the dense matrix $R\in\mathbb{R}^{n\times n}$. However, 
The sparsity of $X_B$ and the low rank of $S_m$ allow for a cheap evaluation of $\|R\|_F$ without
the explicit computation of the dense and large $R$.  To this end, we first write down a quite
standard Arnoldi-type relation for $A$ holding for the space ${\mathbf K}_m(A^{-1},v)$.

\begin{lemma}\label{lemma:arnoldi}
For $v\in\RR^n$, $v\ne 0$, let the columns of $V_m$ be an orthonormal basis of ${\mathbf K}_m(A^{-1},v)$ generated
by the Arnoldi method, so that $A^{-1} V_m \!=\! V_m H_{m}+v_{m+1} h_{m+1,m} e_m^T$. Let
$\eta = \|(I-V_mV_m^T) A v_{m+1}\|$ and
$\widehat v = (I-V_mV_m^T) A v_{m+1}/\eta$. Then 
$$
A V_m = [V_m, \widehat v]\, G_m, \quad {\rm with} \quad G_m = 
\begin{bmatrix}
      I_m & V_m^TAv_{m+1} \\
      0 & \eta\\
      \end{bmatrix}\begin{bmatrix}
       H_m^{-1}\\
      -h_{m+1,m}e_m^TH_m^{-1}\\
      \end{bmatrix} \in \RR^{(m+1)\times m}  .
$$
\end{lemma}

\begin{proof}
Consider the Arnoldi relation
$A^{-1}V_m=V_{m+1}\underline{H}_m=V_mH_m+v_{m+1}h_{m+1,m}e_m^T$,
where $\underline{H}_m\in\mathbb{R}^{(m+1)\times m}$, $\left(\underline{H}_m\right)_{i,j}=h_{i,j}$,
collects the orthogonalization coefficients stemming from the Arnoldi
procedure in lines \ref{lin.solves}--\ref{endArnoldi} in Algorithm \ref{iterativeZ}. Premultiplying by $A$ and 
postmultiplying by $H_m^{-1}$ we get
$$AV_m=V_mH_m^{-1}-Av_{m+1}h_{m+1,m}e_m^TH_m^{-1}=[V_m, Av_{m+1}]\begin{bmatrix}
                                                                H_m^{-1}\\
                                                                -h_{m+1,m}e_m^TH_m^{-1}\\
                                                               \end{bmatrix}.
$$
Let $\eta\,\widehat v:=Av_{m+1}-V_mV_m^TAv_{m+1}$ where $\eta=\|Av_{m+1}-V_mV_m^TAv_{m+1}\|_2$. Then
$$Av_{m+1}=\eta\,\widehat v+V_mV_m^TAv_{m+1}=[V_m,\widehat v]\begin{bmatrix}
                                                             V_m^TAv_{m+1}\\
                                                             \eta\\
                                                            \end{bmatrix},
$$
so that
\begin{eqnarray*}
AV_m&=&[V_m, Av_{m+1}]\begin{bmatrix}
       H_m^{-1}\\
      -h_{m+1,m}e_m^TH_m^{-1}\\
      \end{bmatrix}
      \\
      \\
     %&=&[V_m,\widehat v]\underbrace{\begin{bmatrix}
      &=&[V_m,\widehat v]\begin{bmatrix}
      I_m & V_m^TAv_{m+1} \\
      0 & \eta\\
      \end{bmatrix}\begin{bmatrix}
       H_m^{-1}\\
      -h_{m+1,m}e_m^TH_m^{-1}\\
      \end{bmatrix}=[V_m,\widehat v]\,G_m,
     %\end{bmatrix}}_{=:G_m}=[V_m,\widehat v]\,G_m,\\
      \end{eqnarray*}
where $G_m\in\mathbb{R}^{(m+1)\times (m+1)}$ and $W_m:=[V_m,\widehat v]$ has orthonormal columns by construction.
\end{proof}

\vskip 0.1in

\begin{proposition}\label{prop:res}
With the notation of Lemma~\ref{lemma:arnoldi}, let $W_m = [V_m, \widehat v]$ and
$S_m=V_m\left(\Pi_me^{-\tau \Psi_m}W \Theta^{1/2}\right) =: V_m \Delta_m$. Moreover,
let $R_B=AX_B+X_BA-D$ and $\gamma= \|R_B\|_F$.
Then
$$
\|R\|^2 = \gamma^2 + \|J_m\|_F^2 + 
2\, \mbox{trace}\left(J_m\left(W_m^T R_B W_m\right)\right),
$$
where $J_m = \begin{bmatrix}
                         \begin{array}{c|}
                          I_m \\
                          0\\
                         \end{array}\;\;
                         G_m\\
                        \end{bmatrix}
\begin{bmatrix}
             0 & \Delta_m\Delta_m^T\\
             \Delta_m\Delta_m^T & 0\\
            \end{bmatrix}       
       \begin{bmatrix}
             \begin{array}{c|}
              I_m \\
              0\\
              \end{array}
              \;\;G_m\\
               \end{bmatrix}^T \in \RR^{(m+1)\times (m+1)}$.
\end{proposition}

\vskip 0.1in

\begin{proof}
Recalling that 
$\|G+H\|_F^2=\|G\|_F^2+\|H\|_F^2+2\langle G,H\rangle_F$, it holds
%%
%{\small
$$
\begin{array}{rll}
 \|R\|_F^2&=&\|A(X_B+S_mS_m^T)+(X_B+S_m S_m^T)A-D\|_F^2\\
% &&\\
 &=&\|AS_mS_m^T+S_mS_m^TA\|_F^2+\|AX_B+X_BA-D\|_F^2\\
 && +2\langle AS_mS_m^T+S_mS_m^TA,AX_B+X_BA-D\rangle_F.\\ 
%  &&\\
%  &=&\|(AX_m)X_m^T\|_F^2+\|X_m(X_m^TA)\|_F^2+2\langle X_m(X_m^TA), (AX_m)X_m^T\rangle_F \\
%  && + \|AX_B+X_BA-D\|_F^2+2\left(\langle (AX_m)X_m^T,AX_B)\rangle_F\right.\\
%  && +\langle (AX_m)X_m^T,X_BA)\rangle_F-\langle (AX_m)X_m^T,D)\rangle_F+
%  \langle X_m(X_m^TA),AX_B)\rangle_F\\
%  &&+\left.\langle X_m(X_m^TA),X_BA)\rangle_F-\langle X_m(X_m^TA),D)\rangle_F\right)\\
%  &&\\
%  &=& \|AX_B+X_BA-D\|_F^2+2\left( \|(AX_m)X_m^T\|_F^2\right. \\
%  &&+ \langle X_m(X_m^TA), (AX_m)X_m^T\rangle_F+\langle (AX_m)X_m^T,AX_B)\rangle_F\\
%  && +\langle (AX_m)X_m^T,X_BA)\rangle_F-\langle (AX_m)X_m^T,D)\rangle_F+
%  \langle X_m(X_m^TA),AX_B)\rangle_F\\
%  &&+\left.\langle X_m(X_m^TA),X_BA)\rangle_F-\langle X_m(X_m^TA),D)\rangle_F\right).\\
 \end{array}
$$
%}
%%
The banded matrix $R_B=AX_B+X_BA-D$ and its Frobenius norm can be computed once for all at the beginning of 
Algorithm \ref{iterativeZ}. The computation of the additional two terms can be cheaply carried out
in $\mathcal{O}(sn)$ flops. We first focus on the matrix $AS_mS_m^T+S_mS_m^TA$.
Denoting $\Delta_m:=\Pi_me^{-\tau \Psi_m}W \Theta^{1/2}$, we have
\begin{align}\label{res1}
AS_mS_m^T+S_mS_m^TA%AV_m\left(\Pi_me^{-\tau \Psi_m}\Upsilon\right)\left(\Upsilon^T\Pi_me^{-\tau \Psi_m^T}\right)V_m^T
%+V_m\left(\Pi_me^{-\tau \Psi_m}\Upsilon\right)\left(\Upsilon^T\Pi_me^{-\tau \Psi_m^T}\right)V_m^TA\\
=[V_m, AV_m]\begin{bmatrix}
             0 & \Delta_m\Delta_m^T\\
             \Delta_m\Delta_m^T & 0\\
            \end{bmatrix}\begin{bmatrix}
            V_m^T \\
            V_m^TA\\
            \end{bmatrix}.
\end{align}

Using Lemma~\ref{lemma:arnoldi} we have
$$ AS_mS_m^T+S_mS_m^TA=W_m\underbrace{\begin{bmatrix}
                         \begin{array}{c|}
                          I_m \\
                          0\\
                         \end{array}\;\;
                         G_m\\
                        \end{bmatrix}
\begin{bmatrix}
             0 & \Delta_m\Delta_m^T\\
             \Delta_m\Delta_m^T & 0\\
            \end{bmatrix}       
       \begin{bmatrix}
             \begin{array}{c|}
              I_m \\
              0\\
              \end{array}
              \;\;G_m\\
               \end{bmatrix}^T}_{=:J_m}W_m^T,
$$
so that
$$\|AS_mS_m^T+S_mS_m^TA\|_F^2=\|J_m\|_F^2,$$
and only matrices of order (at most) $m+1$ are involved in the computation of this norm.
Concerning the computation of $\langle AS_mS_m^T+S_mS_m^TA,AX_B+X_BA-D\rangle_F$ we have
\begin{eqnarray*}
%\langle AX_mX_m^T+X_mX_m^TA,AX_B+X_BA-D\rangle_F&=&\mbox{trace}(W_mJ_mW_m^T(AX_B+X_BA-D))\\
%&=&\mbox{trace}(J_m (W_m^T(AX_B+X_BA-D) W_m),
\langle AS_mS_m^T+S_mS_m^TA,R_B\rangle_F&=&\mbox{trace}(W_mJ_mW_m^T R_B)
=\mbox{trace}(J_m W_m^T R_B W_m),
\end{eqnarray*}
and, similarly to $K_m$ and $D_m$,
the matrix $W_m^T R_B W_m\in\mathbb{R}^{(m+1)\times(m+1)}$ requires only the two matrix-vector products
$W_m^T R_B [v_m,\widehat v]$ to be updated at each iteration.
\end{proof}

Although the computation of the residual norm costs $\mathcal{O}(sn)$ flops at each iteration,
lines \ref{res}--\ref{stop} still remain among the most expensive steps
of the overall procedure for solving \eqref{main.eq}
and they are thus performed periodically, say every $d$ iterations.

\begin{remark}
The trace appearing in Proposition \ref{prop:res}
can be carefully computed by further exploiting the trace properties and the
definition of $J_m$. Nonetheless, in finite precision arithmetic cancellations
might occur, so that additional care should be taken in case a very small
residual tolerance - below the square root of machine precision - is selected. 
We did not experience this problem in our numerical tests.
\end{remark}

%%%%%%%%%%%%%%%%%%%%%%%%%%%%%%%%%%%%%%%%%%%%%%%%%%%%%%%%%%%%%%%%%%%%%%%%%%%%%%%%%%%%%%%%%%%%%%%%%%%%%%%%%%%%%%%%%%%%%%%%%%%
\subsection{Complete numerical procedure and the choice of $\tau$}\label{tau} 
The algorithm we propose, hereafter called
{\sc lyap\_banded}, approximates the solution $X$ to (\ref{main.eq}) as $X\approx X_B + S_m S_m^T$ 
%by the pair $(X_B,S_m)$
 where $X_B$ is banded
and $S_m$ is low rank. It is important to realize that unless $\tau \to +\infty$,
the entries of $S_mS_m^T$ contribute in a significant way towards the solution, and in particular
to the nonzero entries of the leading banded part of $X$. Indeed, even
assuming that $X_B$ is exact, that is $X_B = X(\tau)$, we obtain
%{\color{red}
\begin{eqnarray}\label{eqn:error}
e^{-2\tau\lambda_{\max}(A)}\leq\frac{\|X - X_B \|}{\|X\|} \le e^{-2\tau\lambda_{\min}(A)},
\end{eqnarray}
since 
$\|X - X_B \|=\|e^{-\tau A}X e^{-\tau A} \|\leq\|e^{-\tau A}\|^2\|X\|=e^{-2\tau\lambda_{\min}(A)}\|X\|$,
and
$\|e^{-\tau A}X e^{-\tau A} \|\geq\frac{\|X\|}{\|e^{\tau A}\|^2}=e^{-2\tau\lambda_{\max}(A)}\|X\|.$

The performance of {\sc lyap\_banded} crucially depends on the choice of $\tau$. % is crucial in the solution process.
Indeed, a large $\tau$ corresponds to a wider bandwidth of $X(\tau)$ and thus
to a possibly too wide $\beta_{X_B}$. On the other hand, 
 Proposition~\ref{Cor.IItau} says that $e^{-\tau A}Xe^{-\tau A}$ is numerically low rank if $\tau\rightarrow+\infty$. 
Therefore, if the selected value of $\tau$ is too small then the numerical rank of $e^{-\tau A}Xe^{-\tau A}$ 
may be so large that an accurate low rank approximation is hard to determine;
%the rank of $X_m$ might have to be too large  seriously large if a small $\tau$ is employed. 
see Table~\ref{IllEx2} in section~\ref{Numerical_Examples}.
A trade-off between the bandwidth of $X_B$ and the rank of $S_m$ has to be sought.
To make the action of $e^{-\tau A}$ scaling-independent, and
without loss of generality, equation \eqref{main.eq} can be scaled by $1/\lambda_{\min}(A)$, and this is
done in all our experiments.
This seemed to also speed-up the computation of the adaptive quadrature formula.
%Many strategies can be pursued 

To automatically compute a suitable value of $\tau$  we proceed as follows.
Intuitively, we fix a maximum value for $\beta_{X_B}$ and compute the corresponding $\tau$ by using
the decay estimate of Theorem~\ref{Th:decay_exp} applied to $X(\tau)$.
% , we try to estimate the decay rate in the components of $\mathcal{I}_\tau$. 
%Notice that here we are not interested in the actual decay rate but we want to 
%estimate 
If $X(\tau)$ is approximated by the Gauss-Lobatto quadrature formula 
\eqref{gauss_lobatto}, the decay in its off-diagonal entries
 can be estimated by that of $e^{-\tau A}De^{-\tau A}$ 
(for $i=\ell$, $x_i=1$ and $t_i=\tau$ in \eqref{gauss_lobatto}). Note that
 according to Theorem~\ref{Th:decay_exp}, 
the entries of $e^{-\tau A}$ contribute the most to the bandwidth of $e^{-tA}$, $t\in [0, \tau]$
away from the main diagonal, and thus to
%matrix presents the slowest decay of $e^{-tA}$ for $t\in [0,\tau]$, that is, it is the only matrix which contributes to the most external diagonals
the right-hand side of \eqref{gauss_lobatto}. 
In addition, following the discussion at the beginning of section~\ref{Ill_conditionedA},
the multiplication by $D$ does not seem to dramatically influence the final
bandwidth of $e^{-\tau A}De^{-\tau A}$.
%the effect of the banded $D$ is not 
%the matrix $D$ just mitigates the decay of $e^{-\tau A}$ and we can thus focus on the latter. 
Let us thus focus on the first column of $e^{-\tau A}$.
To apply Theorem~\ref{Th:decay_exp} to $e^{-\tau A}$
we fix a value $\beta_{\max}\in\mathbb{N}$ and 
define $\bar\xi:=\lceil |\beta_{\max}-1|/\beta_A\rceil$. 
For\footnote{We recall that for the scaled problem, $\lambda_{\min}(A)=1$,
however for the sake of generality we prefer not to substitute its value.}
$\rho=(\lambda_{\max}(A)-\lambda_{\min}(A))/4$ 
%(recall that for the scaled problem, $\lambda_{\min}(A)=1$,
%however for the sake of generality we prefer not to substitute its value)
%By applying Theorem \ref{Th:decay_exp}, for $\rho \tau\geq 1$ 
and $\sqrt{4\rho \tau }\leq \bar\xi\leq 2\rho \tau$, we have
{%\color{red}
%%
% \begin{align}
\begin{eqnarray}\label{bound_tau1}
  |(e^{-\tau A})_{\beta_{\max},1}|&\leq e^{-\tau \lambda_{\min}(A)}|(e^{-\tau (A-\lambda_{\min}(A)I)})_{\beta_{\max},1}|
  \leq 10\,e^{-\frac{\bar\xi^2}{5\rho \tau}}e^{-\tau \lambda_{\min}(A)}.  
\end{eqnarray}
% \end{align}
  %%
  Similarly, for $\bar\xi\geq 2\rho \tau$,
  \begin{equation}\label{bound_tau2}
|(e^{-\tau A})_{\beta_{\max},1}|\leq 
10\frac{e^{-\rho \tau}}{\rho \tau}\left(\frac{e\rho \tau}{\bar\xi}\right)^{\bar\xi}
  e^{-\tau \lambda_{\min}(A)}.   
  \end{equation}
Our aim is to estimate for which $\tau$ the quantity 
$|(e^{-\tau A})_{\beta_{\max},1}|$ is not negligible while the components from $\beta_{\max}+1$ up
to $n$ in the same column can be considered as tiny.
Since we would like to have a reasonably large value of $\tau$ while
 maintaining $\beta_{\max}$ moderate, we only consider the bound 
\eqref{bound_tau1} in our strategy. Indeed, \eqref{bound_tau2} requires $\bar\xi\geq 2\rho \tau$, 
that is a very large $\beta_{\max}$, to obtain a sizable value of $\tau$. Fixing a threshold $\epsilon_\tau$,
we can compute $\tau$ as
\begin{equation}\label{tau_condition}
  \tau_{opt}=\mbox{argmin}\,\{t\geq0\mbox{ s.t. }|(e^{-t A})_{\beta_{\max},1}|\geq\epsilon_\tau\}.
\end{equation}
In \cite{Benzi2015} it has been shown that the bounds in Theorem \ref{Th:decay_exp}
are rather sharp, leading to correspondingly sharp bounds \eqref{bound_tau1}--\eqref{bound_tau2}.
 This allows us to save computational costs by replacing (\ref{tau_condition}) with
% save We thus prefer to compute $\tau$ as
%%
$$ 
\tau:=\mbox{argmin}\,\{t\geq 0\mbox{ s.t. }10\,e^{-\frac{\bar\xi^2}{5\rho t}}e^{-t \lambda_{\min}(A)}\geq
\epsilon_\tau\}
 \approx \tau_{opt},
 %\mbox{argmin}\,\{t\geq0\mbox{ s.t. }|(e^{-t A})_{\beta_{\max},1}|\geq\epsilon_\tau\}.
$$
%%
%{\color{red} Qui non si usa il $\tau$ argmin, ma quello ottenuto dalla stima...quindi il $\tau$
%qui sopra non \`e lo stesso del $\tau$ qui sotto. Cio\`e ci sara' una disuguaglianza qui sotto. Quale?}
and a direct computation shows that 
\begin{equation}\label{choice_tau}
\tau=\frac{1}{10\rho\lambda_{\min}(A)}\left(-5\rho\log(\epsilon_\tau/10)-\sqrt{25\rho^2\log^2(\epsilon_\tau/10)-
20\rho\lambda_{\min}(A)\bar\xi^2}\right). 
\end{equation}
To clarify the discussion, let us consider the vector-valued function $f:\mathbb{R}\rightarrow \mathbb{R}^n$,
$f_i(t):=10\,e^{-\frac{\xi_i^2}{5\rho t}}e^{-t \lambda_{\min}(A)}$, $\xi_i=\lceil |i-1|/\beta_A\rceil$,
$i=1,\ldots,n$. Choosing $\tau$ as in \eqref{choice_tau} ensures that  
$f_{\bar\xi+1}(\tau)\geq\epsilon_\tau$ whereas $f_{\bar\xi+1+k}(\tau)<\epsilon_\tau$, $k>0$, so that also
$|(e^{-\tau A})_{\bar\xi+1+k,1}|<\epsilon_\tau$.
A graphical description is provided in the following Example \ref{tau_ex}.

 \begin{figure}[htb]
\begin{center}
  \includegraphics[scale=0.6]{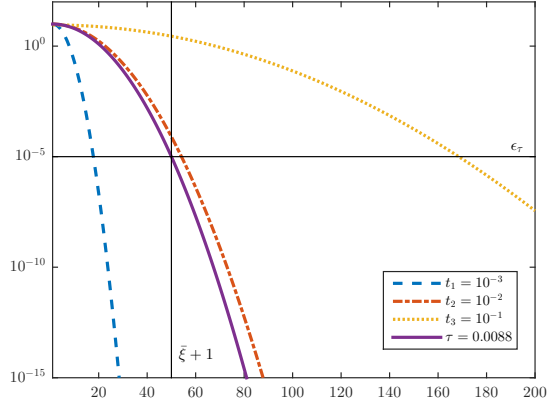}%
\caption{$f(t)$ for different values of $t$ and $n=200$.}\label{Fig.tau}
\end{center}
\end{figure}

\begin{table}[htb]
 \centering
 \begin{tabular}{r|rr}
  & $t=t_1$ & $t=\tau$ \\
  \hline
  $f_{\bar\xi}(t)$ & $1.27\cdot10^{-4}$ & $1.74\cdot10^{-5}$ \\
  $f_{\bar\xi+1}(t)$ & $7.95\cdot10^{-5}$ & $1\cdot10^{-5}$ \\
  $f_{\bar\xi+2}(t)$ & $4.90\cdot10^{-5}$ & $5.66\cdot10^{-6}$ \\
 \end{tabular}\caption{Example \ref{tau_ex}. Values of $f_{\bar\xi+k}(t)$, $k=0,1,2$, $t=t_1,\tau$.}\label{Tab.tau}
\end{table}

\begin{example}\label{tau_ex}
{\rm
 Consider $A=L/\lambda_{\min}(L)$ where 
 $L=\mbox{tridiag}(-1,\underline{2},-1)\in\mathbb{R}^{n\times n}$, $n=200$.
  Figure~\ref{Fig.tau} displays the function 
 $f$ for different values of $t$ and for $\tau$ computed by \eqref{choice_tau} where $\epsilon_\tau=10^{-5}$ and 
 $\beta_{\max}=50$. The range of the $y$-axis is restricted to
$[10^{-15},10^{2}]$ so as  to better appreciate the trend of the 
 largest entries of $f(t)$.
% However, this implies that the smallest values of the function are not reported for $t=t_1,t_2,\tau$.
%
 Since $\beta_{max}=50$ and $\beta_A=1$, it holds that $\bar\xi=49$. For $t=t_1$, 
 $f_{\bar\xi+1}(t_1)=1.11\cdot10^{-50}<\epsilon_\tau$
 so that $t_1$ is not a useful value for our purpose.
 On the other hand, for $t=t_3$, $f_{\bar\xi+1}(t_3)=2.79\geq\epsilon_\tau$ but also many of the subsequent
 values satisfy
 $f_{\bar\xi+1+k}(t_3)\geq\epsilon_\tau$. This may lead to an undesired large bandwidth 
 when the rational approximation to $e^{-t_3A}$ is actually computed.
 We obtain a similar behavior for $f(t)$ when $t=t_2,\tau$, but only for $t=\tau$ we have   
 that $f_{\bar\xi+1}(\tau)\geq\epsilon_\tau$, whereas it holds that $f_{\bar\xi+1+k}(\tau)<\epsilon_\tau$,
 as illustrated in Table~\ref{Tab.tau}.
}
\end{example}

The overall procedure is summarized in the following algorithm.
 
\begin{algorithm}
\setcounter{AlgoLine}{0}
\NoCaptionOfAlgo
\caption{{\sc lyap\_banded}: Numerical approximation $X\approx X_B+S_m S_m^T$.}%\label{overall_alg}
\SetKwInOut{Input}{Input}\SetKwInOut{Output}{Output}
%%%%%%%%%%% INPUT %%%%%%%%%%%
\Input{$A\in\mathbb{R}^{n\times n},$ $A$ SPD, $D\in\mathbb{R}^{n\times n}$, 
$\beta_{\max},\nu, m_{\max}\in\mathbb{N}$, $\epsilon_{\tau}$, $\epsilon_{B}$, $\epsilon_{quad}$,
$\epsilon_{Res}$}%,\epsilon_{It}>0$}
%%%%%%%%%%% OUTPUT %%%%%%%%%%%
\Output{$X_B\in\mathbb{R}^{n\times n}$, $S_m\in\mathbb{R}^{n\times s}$, $s\ll n$}
%%%%%%%%%%%%%%%%%%%%%%%%%%%%%%%%%%%
\BlankLine
\nl Compute $\tau$ by \eqref{choice_tau} \\
\nl Compute $X_B$ by Algorithm \ref{X_B_alg} \\
\nl Compute $S_m$ by Algorithm \ref{iterativeZ} \\
\end{algorithm}
 
Notice that approximations to the
 extreme eigenvalues of $A$ are necessary to be able to
 compute $\tau$ via \eqref{choice_tau}. In all our numerical examples, approximations to 
$\lambda_{\min}(A)$ and $\lambda_{\max}(A)$ were obtained by means of the Matlab function {\tt eigs}.

 Finally, since the strategy adopted for choosing $\tau$ is related to the computation of the banded part
of the solution, we suggest to set $\epsilon_\tau=\epsilon_{quad}$.
%%%%%%%%%%%%%%%%%%%%%%%%%%%%%%%%%%%%%%%%%%%%%%%%%%%%%%%%%%%%%%%%%%%%%%%%%%%%%%%%%%%%%%%%%%%%%%%%%%%%%%%%%%%%%%%%%%%%%%%%%%%%%%
\section{Numerical solution of the Sylvester equation}\label{Sylvester_eqs}
The procedure proposed in the previous sections can be extended to the case of the
following {\it Sylvester} equation,
\begin{equation}\label{main_Sylv}
 AX+XB=D,
\end{equation}
with $A\in\mathbb{R}^{n_A\times n_A}$, $B\in\mathbb{R}^{n_B\times n_B}$ banded and SPD, and $D\in\mathbb{R}^{n_A\times n_B}$
banded. For ease of presentation we consider the case $n=n_A=n_B$, while different $n_A,n_B$ could be considered as well.
Once again, the selection of which numerical procedure should be used between those discussed in
the previous sections depends
on $\kappa({\cal A})$, where here $\mathcal{A}=B\otimes I+I\otimes A$. In this case,
$\kappa(\mathcal{A})=(\lambda_{\max}(A)+\lambda_{\max}(B))/(\lambda_{\min}(A)+\lambda_{\min}(B))$,
therefore the magnitude of $\kappa(\mathcal{A})$ depends on the relative size of
the extreme eigenvalues of $A$ and $B$.

If $\mathcal{A}$ is well-conditioned, Algorithm \ref{CG_matrix} can be applied with 
straightforward modifications in lines 1 and 2. Notice that, even if $D$ is 
symmetric, none of the {\sc cg} iterates is symmetric so that the memory-saving strategies and computational
tricks discussed in section~\ref{Well-conditioned} 
cannot be applied. Nevertheless, the bandwidth of the iterates still 
grows linearly with the number of iterations. 

 \begin{proposition}%\label{THM1}
 If $X_0=0$,  
  all the iterates generated by {\sc cg} applied to equation \eqref{main_Sylv} are banded matrices and, in particular,
  $$\begin{array}{ll}
   \beta_{W_k}\leq k\max(\beta_A,\beta_B)+\beta_{D},&  \beta_{X_k}\leq (k-1)\max(\beta_A,\beta_B)+\beta_{D},\\
   &\\
   \beta_{R_k}\leq k\max(\beta_A,\beta_B)+\beta_{D},&   \beta_{P_k}\leq k\max(\beta_A,\beta_B)+\beta_{D}.\\
  \end{array}$$
  %where $\beta_T$ denotes the bandwidth of a given matrix $T$.
 \end{proposition}

 \begin{proof}
  The same arguments of the proof of Theorem \ref{THM1} can be applied noticing that the bandwidth of the 
matrix $W_k=AP_k+P_kB$ is such that $\beta_{W_k}\leq\max(\beta_A,\beta_B)+\beta_{P_k}$.
 \end{proof}

 If $\cal A$ is ill conditioned, Algorithm  {\sc lyap\_banded} can be generalized to handle the new setting.
 The solution $X$ can be written as (see, e.g., \cite{Simoncini2016})
 \begin{equation}\label{SolSylv}
 X=\int_0^{+\infty}e^{-tA}De^{-tB}dt=\int_0^\tau e^{-tA}De^{-tB}dt+\int_\tau^{+\infty}e^{-tA}De^{-tB}dt.  
 \end{equation}
 A procedure similar to Algorithm \ref{X_B_alg} can be applied to approximate the first integral.
 Clearly, the presence of two different matrix exponentials increases the computational cost of the method as 
 two approximations $\widehat R_\nu(t_iA)$, $\widehat R_\nu(t_iB)$ have to be computed at each node. 
 
To approximate the second integral addend in \eqref{SolSylv} we can generalize Algorithm~\ref{iterativeZ}.
Taking into account the presence of two coefficient matrices, 
a left and a right space need to be constructed, namely
 $\mathbf{K}_m(A^{-1},v)$, $\mathbf{K}_m(B^{-1},w)$, as it is customary in projection methods for 
 Sylvester equations. 
 
 The choice of $\tau$ may be less straightforward in case of \eqref{main_Sylv}. If $A$ and $B$ have similar condition numbers,
 we suggest to still compute $\tau$ by 
 \eqref{choice_tau} but replacing $\lambda_{\min}(A)$ by $\lambda_{\min}(C)$, where $C$ is the matrix with the 
 widest bandwidth\footnote{Also the computation of $\rho$ in \eqref{choice_tau} will change accordingly.}
 between $A$ and $B$.
 
%two matrix-matrix products $W_k=AP_k+P_kB$ are required at each 
%witeration. 

%%%%%%%%%%%%%%%%%%%%%%%%%%%%%%%%%%%%%%%%%%%%%%%%%%%%%%%%%%%%%%%%%%%%%%%%%%%%%%%%%%%%%%%%%%%%%%%%%%%%%%%%%%%%%%%%%%%%%%%%%%
\section{Numerical examples}\label{Numerical_Examples}
In this section we present numerical experiments illustrating the effectiveness of 
the method {\sc lyap\_banded}.

 Banded matrices are a particular example of
$\mathcal{H}$-matrices, so that algorithms specifically
  designed to deal with this kind of structure could be employed in solving 
equation~\eqref{main.eq}.  The very low memory requirements 
%to store $\mathcal{H}$-matrices 
is one of the features of the $\mathcal{H}$-format. Although we are not going to
implement an ad-hoc routine for $\mathcal{H}$-matrices computations, in Example \ref{Ex.4} we compare the 
memory requirements to store the pair $(X_B,S_m)$
  with those requested to store a comparably accurate approximate solution obtained
in $\mathcal{H}$-format. To this end,
  we use the {\tt hm-toolbox}\footnote{{\tt https://github.com/numpi/hm-toolbox}} developed 
while writing \cite{Massei2017}; to the best of our knowledge, this is the only available 
Matlab toolbox for $\mathcal{H}$-format computation.
 In particular, in the {\tt hm-toolbox} a  % is implemented a 
subclass of the set of $\mathcal{H}$-format representations - sometimes called Hierarchically Off-Diagonal
Low-Rank (HODLR) format - is implemented; see, e.g., \cite[Chapter 3]{massei2017exploiting} for more details.

%To the best of our knowledge, \cite{Haber2016} is the only paper that deals with large-scale Lyapunov equation with 
%banded data and well-conditioned $A$, whereas no solvers are available in the literature in case of ill-conditioned $A$.
%When Algorithm \ref{CG_matrix} is applied, the stopping criterion is based on the relative residual norm and the 
%threshold is $\epsilon=10^{-8}$.
All results were obtained with Matlab R2015a on a Dell machine with two 2GHz processors and 128 GB of RAM.
All reported experiments use the parameter settings in Table~\ref{tab:params}.

{\small
\begin{center}
\begin{table}[htb]
\begin{tabular}{|l|l|}
\hline
$\epsilon_{res}=$  $10^{-3}$ & relative residual stopping tol ({\sc cg}, {\sc lyap\_banded}) \\
$m_{\max}=  2000$       & max number of iterations ({\sc cg}, {\sc lyap\_banded}) \\
$(\epsilon_\tau, \beta_{\max})= (10^{-5}, 500)$ & setting for the computation of $\tau$ in {\sc lyap\_banded} \\
$(\nu, \epsilon_B, \epsilon_{quad})=  (6, 10^{-5}, 10^{-5})$ & truncation and approximation parameters for $X_B$\\
\hline
\end{tabular}
\caption{\label{tab:params}}
\end{table}
\end{center}
}

%\vskip 0.1in

\begin{num_example}\label{Ex.3}
{\rm
We consider the symmetric
tridiagonal matrix $A\in\mathbb{R}^{n\times n}$ (thus $\beta_A=1$) stemming from the
 discretization by centered finite differences of the 1D differential operator
$$
\mathcal{L}u=-\frac{1}{\gamma}\left(e^{x}u_{x}\right)_x+\gamma u,
$$
    %%
%   $$\mathcal{L}u=-\left(e^{x}u_{x}\right)_x,$$
   %%
   in $\Omega=(0,1)$ with zero Dirichlet boundary conditions. The matrix 
$A$ is asymptotically ill-conditioned due to the second 
   order term of the operator, and $\kappa(A)$
   grows with $n$. The parameter $\gamma\in\mathbb{R}$ is used to vary the condition number of $A$.   
  %We thus define $A:=L+\gamma_k I$ where $\gamma_k\in\mathbb{R}$ is chosen to control the condition number of $A$.
  %In particular, $\gamma_k:=(\lambda_{max}(L)-10^k\lambda_{min}(L))/(10^k-1)$ so that $\kappa(A)=10^k$.
 The right-hand side $D$ of \eqref{main.eq} is a diagonal matrix (thus $\beta_D=0$) with
   uniformly distributed random diagonal entries. 
% We test Algorithm \ref{CG_matrix} with $\epsilon_{CG}=10^{-3}$ and
We ran {\sc lyap\_banded} for different values of $n$ and $\kappa(A)$ and compare its performance
with that of Algorithm \ref{CG_matrix}. In {\sc lyap\_banded} the parameter $\tau$ is computed
with the parameters set in Table \ref{tab:params}.
%considered values of $\epsilon_\tau=10^{-5}$ and $\beta_{\max}=500$.
%  All the thresholds and parameters of {\sc lyap\_banded} are set as follows. We compute $\tau$ by means
%  of \eqref{choice_tau} with $\epsilon_\tau=10^{-5}$ and $\beta_{\max}=500$.
% For computing $X_B$ we consider $\nu=7$, $\epsilon_{B}=10^{-7}$
% for all $i,j$, and $\epsilon_{spy}=10^{-5}$. 
%%%
%For the low-rank part $S_m$, 
%we set {\color{red}$\epsilon_{It}=10^{-2}$ (stagnation flag).}
%%%
% and $m_{\max}=1000$ (maximum iteration number). 
The relative residual norm $\|R\|_F/\|D\|_F$ is computed every $d=10$ iterations.
Table \ref{IllEx} collects the results as $n$ and $\gamma$ vary. 
 %When possible, we report also the 
 %relative residual norm $\|R\|_F/\|D\|_F=\|A(X_B+X_L'X_L''^T)+(X_B+X_L'X_L''^T)A-D\|_F/\|D\|_F$ to evaluate the accuracy
 %of the computed solution.
 %%
 \begin{table}[htp]
\centering
{\scriptsize
\begin{tabular}{rrr|rrrr|rrrrr}
$n$ & $\gamma$ & $\kappa(A)$ & \multicolumn{4}{c|}{{\sc cg} (Algorithm \ref{CG_matrix})} & 
\multicolumn{5}{c}{{\sc lyap\_banded}} \\
  &  &  &Its. & $\beta_X$  &  Time  &  Res.& $\tau$ & $\beta_{X_B}$ & $s$ & Time & Res. \\
 %&  &  &Its. & $\beta_X$  &  Time  &  Res.& $\tau$ & $\beta_{X_B}$ & $\mbox{rank}(X_m)$ & Time & Res. \\
  \hline 
\!\!\!\!   $4\,\cdot 10^4$ & 1000 &  6.61e3 &290 &289 &3.77e2 &9.87e-4 &  2.73  & 480  & 7 &1.44e3  &
3.88e-4  \\
        & 500 & 2.68e4 & 583 &582 &1.57e3 &9.92e-4 & 0.56 & 578   &340  &1.63e3  &9.86e-4  \\
        & 200 & 1.72e5 & 1475 &1474 &1.09e4 &9.99e-4 &0.08  &  594 & 366  & 1.66e3 & 9.57e-4  \\
  \hline
 % \hline
\!\!\!\!   $7\,\cdot 10^4$ & 1800 &  6.19e3 & 281 & 280&6.20e2 &9.82e-4 & 2.98  &466 & 7&2.46e3  &
3.22e-4   \\
         & 1000 & 2.02e4 & 507 & 506 &2.02e3 &9.89e-4 & 0.76 &571 &576 &3.38e3 &
         9.89e-4  \\
         & 400 & 1.29e5 & 1277  & 1276 & 1.41e4 & 9.98e-4 &0.11 &592  &632  & 3.79e3  &
         9.56e-4  \\
 % \hline
  \hline
\!\!\!\!   $10^5$ & 2500 &  6.53e3   & 288 &287 &9.11e2 & 9.94e-4 &2.77   &478  & 7  & 3.96e3  &
 3.44e-4   \\
          & 1500 & 1.82e4 & 481 & 480 & 2.56e3 &9.96e-4 &0.84 & 570 & 812 & 6.77e3  &9.73e-4  \\
          & 500 & 1.67e5 &1456  &1455 &2.65e4 &9.96e-4 & 0.08 &  594 & 892 & 7.15e3  & 9.87e-4  \\ 
 
  %\multicolumn{8}{c}{$\beta_{max}=150$ } \\
  %\hline
  %10000 & 300 &  4.72e+3   &  9.40e+4  & 7.48e+1 (193) & 9.43e+1 (445) & 1.69e+2 & 7.76e-5 \\
  %10000 & 100 & 4.44e+4 & 2.94e+4 & 7.46e+1 (195)   & 1.02e+2 (446) & 1.76e+2 & 8.71e-5 \\
  %10000 & 50 & 1.85e+5 & 1.40e+4 & 7.70e+1 (195)   & 9.44e+1 (442) & 1.71e+2 & 8.90e-5 \\
 %\hline
 %\hline
  \end{tabular} 
  \caption{Example \ref{Ex.3}. Results for different values of $n$ and $\gamma$. $s={\rm rank}(S_m)$. 
Time is CPU time in seconds. \label{IllEx}}
}
  \end{table}

 Algorithm \ref{CG_matrix} is very effective up to $\kappa(A)\approx \mathcal{O}(10^4)$, while
for the same $\kappa(A)$ {\sc lyap\_banded} is rather expensive in terms of CPU time
  compared to {\sc cg}. The role of the two methods is reversed for
 $\kappa(A)=\mathcal{O}(10^5)$. In this case, {\sc cg} takes a lot of iterations to
meet the stopping criterion; the costs of {\sc lyap\_banded} grow far less dramatically,
making the method competitive, both
% The performance of Algorithm~\ref{CG_matrix} deteriorates for larger condition numbers:
% when $\kappa(A)=\mathcal{O}(10^5)$ many
%  iterations are needed to converge leading to a high computational effort.
%   In this case, {\sc lyap\_banded} allows as to
%  gain one order of magnitude 
in terms of CPU time and storage demand. The bandwidth obtained by {\sc cg} is lower
than that  obtained by the banded portion in {\sc lyap\_banded} for the smaller
conditions numbers, while the situation is reversed for the largest value of $\kappa(A)$.
  
  Regarding {\sc lyap\_banded},
we notice that for fixed $n$ both $\beta_{X_B}$ and $\mbox{rank}(S_m)$ grow with $\kappa(A)$.
In particular, $\mbox{rank}(S_m)$ is consistently much lower for the first value of $\gamma$
than for the other ones. This can be explained by noticing the quite different value of $\tau$
taken as $\gamma$ varies. This dramatically influences the exponential $\exp(-2\tau)$, 
and thus
the expected error bound for the banded part of the approximation. For instance, for $n=4\,\cdot 10^4$
we obtain

%\vskip 0.05in
\begin{center}
$\tau=2.73, \quad  \exp(-2 \tau)=4.3\,\cdot 10^{-3}$ \\
$\tau=0.56, \quad  \exp(-2 \tau)=3.2\,\cdot 10^{-1}$\\
$\tau=0.08, \quad \exp(-2 \tau)=8.5\,\cdot 10^{-1}$
\end{center}
\vskip 0.05in

Taking into account the error upper bound in (\ref{eqn:error}), we have
$\|X - X_B\| \le \|X-X(\tau)\| + \|X(\tau) - X_B\| \le e^{-2\tau}\|X\| + \|X(\tau) - X_B\|$.
Therefore, if $X_B$ is a good approximation to $X(\tau)$, the leading term in the bound is
$e^{-2\tau}\|X\|$.  For $\tau=2.73$, the small value of $e^{-2\tau}$ shows
that the banded part $X_B$ is
already a good approximation to the final solution, so that a very low rank approximate
solution is sufficient to finalize the procedure. This is not the case for the other
values of $\tau$.

For similar values of $\kappa(A)$, 
  only $\mbox{rank}(S_m)$ is affected by an increment in the problem size. This phenomenon is 
  associated with the strategy we adopt for choosing $\tau$. Indeed, a fixed value $\beta_{\max}$ is employed
  and $\tau$ is computed according to \eqref{choice_tau}; this way $\tau$ only depends on the 
(rescaled) extreme eigenvalues of $A$,
  whose magnitude is similar for comparable $\kappa(A)$.
  Since the $n$ eigenvalues of $A$ seem to spread quite evenly in the interval $[1,\kappa(A)]$ 
  the number $\bar\ell$ of eigenvectors required to get an equally accurate low-rank 
matrix $X_L$ in Corollary~\ref{Cor.IItau} increases with $n$.

 We next set $n=40000$, $\gamma=500$. All the other parameters are as before.
 We vary $\tau$ to study how its choice affects the performance of the algorithm. 
 Results are reported in Table \ref{IllEx2}.
The reference value of $\tau$ (first line in the table) is obtained
%The reference value of $\tau$ (first line in Table \ref{IllEx2}) is obtained
with the default values of the parameters, as in Table \ref{tab:params}, and with
the automatic procedure of section~\ref{tau}.
% In particular, in the first line of Table \ref{IllEx2}, $\tau$ is computed by 
% \eqref{choice_tau} where $\beta_{\max}=500$ and $\epsilon_\tau=10^{-5}$. 
All the other values of $\tau$ are selected as $10^{j}$, $j=-2,\ldots, 1$.
%without any specific criterion.
 
 \begin{table}[htp]
\centering
\begin{tabular}{rrrrr}
  $\tau$ & $\beta_{X_B}$ & $\mbox{rank}(S_m)$ & Time  & Res. \\
  \hline
  0.56 & 578   &340  & 1.63e3  &9.86e-4 \\
  \hline
    0.01  & \,\,92 & 1894 & 4.69e3 & 1.13e-2 \\
    \,\,0.1 & 270   & 861 & 1.43e3 & 9.88e-4 \\
    \,\,\,\,\,1 & 718   & 270 &2.70e3 & 9.89e-4 \\
    \,\,\,10 & 874  & 213 & 5.28e3 & 1.50e-3 \\ 
  \end{tabular} 
  \caption{Example \ref{Ex.3} with $n=40000$ and $\gamma=500$.
Results for different values of $\tau$. \label{IllEx2}}
  \end{table}

As expected, a small $\tau$ leads to a very tight bandwidth of $X_B$ but a too large rank of $S_m$.
 On the other hand, a very large $\tau$ causes an increment in 
the bandwidth of $X_B$ while a very low-rank $S_m$ is 
 computed. Notice that a proper value of $\tau$ is essential also in terms of accuracy of the 
 numerical solution.
 Indeed, for $\tau=0.01$, Algorithm \ref{iterativeZ} stops because the 
maximum number of iterations $m_{\max}=2000$ is reached,
 while for $\tau=10$ a too small residual norm reduction causes a stagnation flag.  
Good performance is obtained for $\tau=0.1,1$,
 although both values lead to larger memory requirements 
than those obtained with $\tau$ computed by \eqref{choice_tau}.
}

\end{num_example}

\vskip 0.1in

\begin{num_example}\label{Ex.4}
{\rm
  We consider the matrix $A\in{\mathbb R}^{n\times n}$ stemming from the
 discretization by centered finite differences of the 1D differential operator
 $${\cal L}(u)=-u_{xx}+\gamma\log(10(x+1))u,$$
   in $\Omega=(0,1)$ with zero Dirichlet boundary conditions and $\gamma>0$. If $\Omega$ is discretized
   by $n$ nodes $(x_1,\ldots,x_n)$,
   %$\{x_i\}_{i=1}^n$,
   we have
 $$
A=-\frac{(n-1)^2}{12}\mbox{pentadiag}(-1,16,-\underline{30},16,-1)+
\gamma\mbox{diag}(\chi_1, \ldots, \chi_n), \quad \chi_j=\log(10(x_j+1))
%\gamma\cdot\mbox{diag}(\log(10(x_1+1)),\ldots,\log(10(x_n+1)))
$$
 (four neighboring points were used for each grid node).
   As in the previous example, the matrix $A$ is asymptotically ill-conditioned and $\gamma$ is chosen to control its 
   condition number, so that $A=A(\gamma)$.
   %and we thus define 
   %$A:=L+\gamma_k I$ where $\gamma_k\in\mathbb{R}$ is chosen to control the condition number of $A$.
  %In particular, $\gamma_k:=(\lambda_{\max}(L)-10^k\lambda_{\min}(L))/(10^k-1)$ so that $\kappa(A)=10^k$.
 The right-hand side $D$ of \eqref{main.eq} is a symmetric tridiagonal matrix with
   uniformly distributed random entries and unit Frobenius norm. Both $A$ and $D$ are banded, 
   with $\beta_A=2$ and $\beta_D=1$.

 \begin{table}[htb]
\centering
{\footnotesize
\begin{tabular}{rrrrrrrr}
 
$n$  & $\gamma$ & $\kappa(A)$ &  $\tau$ & Time $X_B$ ($\beta_{X_B}$) & Time $S_m$ ($s$) & Time Tot. & Res. \\
  \hline 
   $4\, 10^4$ & 5000 & 7.00e5 & 2.07e-2  & 2.45e3 (523)& 3.11e2 (464)  &
   2.77e3 & 9.89e-4  \\
              & 800  & 4.20e6 & 3.45e-3  & 2.28e3 (523) & 3.15e2 (464)  &
              2.60e3 & 9.98e-4  \\
              & 300  & 1.08e7 & 1.32e-3  & 2.24e3 (523) & 3.09e2 (464)
              & 2.55e3 & 9.97e-4 \\
  \hline 
   $7\, 10^4$ & 15000& 7.27e5 & 1.99e-2  & 4.01e3 (523) & 1.54e3 (795)  &
   5.55e3 & 9.86e-4  \\
              & 2000 & 5.27e6 & 2.78e-3  & 4.02e3 (523) & 1.54e3 (794)
              & 5.55e3 & 9.95e-4   \\
              & 800  & 1.28e7 & 1.16e-3  & 3.99e3 (523) & 1.54e3 (793)  & 
              5.53e3 & 9.92e-4  \\
   
   \hline
   $10^5$     & 50000& 4.51e5 & 3.22e-2  & 6.21e3 (523) & 4.47e3 (1124) & 
   1.07e4 & 9.86e-4  \\
              & 5000 & 4.38e6 & 3.34e-3  & 6.08e3 (523)& 4.47e3 (1124) &
              1.04e4 & 9.99e-4 \\
              & 200  & 6.78e7 & 2.13e-4  & 5.90e3 (523) & 4.44e3 (1129) & 
              1.03e4 & 9.77e-4 \\   
  \end{tabular} 
   \caption{Example \ref{Ex.4}. Results for different values of $n$ and $\gamma$. 
The timings reported are in seconds. $s={\rm rank}(S_m)$. \label{IllEx.2}}
  }
  \end{table}
  
  We solve this problem only by {\sc lyap\_banded} as the large $n$'s and the moderate values of $\gamma$ we 
  considered lead to sizeable values of $\kappa(A)$.
  All the thresholds and parameters of the procedure are set as in Table \ref{IllEx}.  
 In Table \ref{IllEx.2} we collect the results as $n$ and $\gamma$ vary.
 We also report the CPU time devoted 
 to the computation of $X_B$ and $S_m$ respectively.

  We notice that in this example, the fixed value $\beta_{\max}$ leads to a constant $\beta_{X_B}$ 
  for all the tested $n$'s. 
  Moreover, for a given $n$, also the rank of the computed 
$S_m$ turns out to be almost independent of $\kappa(A)$. 
This can be intuitively explained by referring to Figure~\ref{Fig.exp_eig}, where the values
of $\exp(-\tau \lambda_j)$ above $10^{-8}$ are plotted for three automatic selections
of $\tau$ - as the operator parameter $\gamma$ changes -  and for
the smallest eigenvalues of $A$. The legend also gives the {\em number} of values above
the threshold, for the given $\tau$. Both the distribution and the number
of eigenvalues of $A=A(\gamma)$ giving an exponential above the threshold $10^{-8}$ are 
approximately the same for all
selections of $\tau$, showing that the automatic selection of $\tau$ well adapts to
the change in the spectrum given by the different $\gamma$'s.

 \begin{figure}[htb]
\begin{center}
  \includegraphics[scale=0.6]{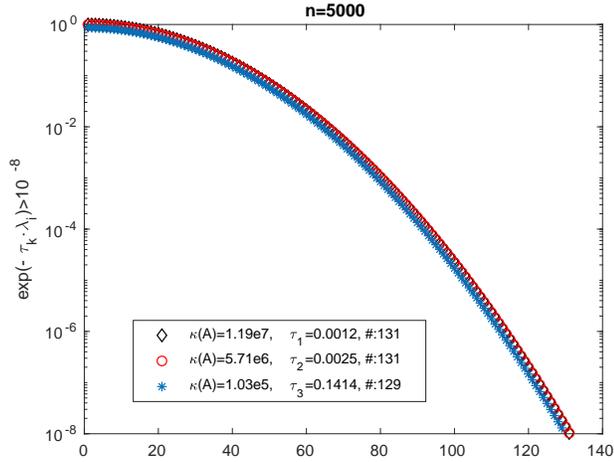}%
\caption{Values of $\exp(-\tau_k \lambda_j)$ above the threshold $10^{-8}$. Larger
eigenvalues of $A$ contribute very little to the value of the exponential.
\label{Fig.exp_eig}}
\end{center}
\end{figure}

 We next compare the storage demand of {\sc lyap\_banded} with that of
an $\mathcal{H}$-format approximation to the solution $X$. We consider a smaller 
problem, $n=5000$, so as to compute 
$X$ by the Bartels-Stewart algorithm (Matlab function {\tt lyap}). %We then obtain the matrix
 The comparison matrix $\widebar X\approx X$ is obtained from $X$}
 by means of the function {\tt hm} available in the Matlab toolbox {\tt hm-toolbox}. 
The parameters for the 
$\mathcal{H}$-format compression are set so as to have a similar residual norm in $\widebar X$
and $(X_B,S_m)$: we set {\tt hmoption('threshold',1e-7)} for $\gamma=200,20$ 
and {\tt hmoption('threshold',1e-8)} for $\gamma=0.2$.
Table~\ref{IllEx.3} collects the results.
    
  \begin{table}[htb]
 \centering
 {\small
 \begin{tabular}{rr|ccc|cc}
   &  &  $X_B$  &  $S_m$ & $(X_B,S_m)$ & $\widebar X$ &  $\widebar X$ \\
  $\gamma$ & $\kappa(A)$ &  Bytes ($\beta_{X_B}$) & Bytes ($s$) & Res. & Bytes  & Res.  \\
 %$\gamma$ & $\kappa(A)$ &  Bytes $X_B$ ($\beta_{X_B}$) & Bytes $S_m$ ($s$) & Res. $(X_B,S_m)$ & Bytes $\widebar X$ & Res. $\widebar X$ \\
   \hline 
      200 & 2.50e5 & 4.29e7 (275) & 4.68e6 (117) & 9.49e-4  & 1.11e7 & 1.48e-4  \\
      20  & 2.09e6 & 4.29e7 (275)  & 4.68e6 (117) & 9.73e-4  & 1.05e7 & 9.98e-4  \\
      0.2 & 1.28e7 & 4.29e7 (275)  & 4.68e6 (117) & 9.23e-4 & 1.09e7 & 6.33e-4 \\
   \end{tabular} 
      \caption{Example \ref{Ex.4}. Results for different values of $\gamma$. 
 $s={\rm rank}(S_m)$.  $\mathtt{\bar X= hm(X)}$. \label{IllEx.3}}
}
   \end{table}

For the same level of residual accuracy, the numbers in Table \ref{IllEx.3} show that
the memory requirements for $\widebar X$ are of
   the same order of magnitude as those for storing $(X_B,S_m)$, suggesting that the splitting procedure we propose
   works rather well in terms of memory demands.

\end{num_example}

%%%%%%%%%%%%%%%%%%%%%%%%%%%%%%%%%%%%%%%%%%%%%%%%%%%%%%%%%%%%%%%%%%%%%%%%%%%%%%%%%%%%%%%%%%%%%%%%%%%%%%%%%%%%%%%%%
\section{Conclusions}\label{Conclusions}
In this paper we have addressed the solution of large-scale Lyapunov equations with banded
symmetric data and positive definite coefficient matrix $A$.
In case of well-conditioned $A$, the numerical solution can be satisfactorily
approximated by a banded matrix, so that the matrix-oriented {\sc cg} method has been 
shown to be a valid candidate for its computation.

If the coefficient matrix is ill-conditioned, no banded good approximation can be 
determined in general.
However, we showed that the solution $X$ can be represented in terms
of the splitting $X_B+S_mS_m^T$, with $X_B$ banded and $S_m$ low-rank, 
and an efficient procedure for computing the pair
$(X_B,S_m)$ was presented. Our preliminary numerical results show that the new method
is able to compute a quite accurate approximate solution, and that the tuning of
the required parameters is not too troublesome.

Both the derivation and the algorithm were extended to the case of Sylvester
equations with banded symmetric data and positive definite coefficient matrices.

%%%%%%%%%%%%%%%%%%%%%%%%%%%%%%%%%%%%%%%%%%%%%%%%%%%%%%%%%%%%%%%%%%%%%%%%%%%%%%%%%%%%%%%%%%%%%%%%%%%%%%%%%%%%%%%%%%%%%

\section*{Appendix}
Here we report the algorithm presented in section~\ref{Implementation details} for solving the linear system~(\ref{LU}). 
%%LU
\begin{algorithm}
\setcounter{AlgoLine}{0}
%\algsetup{linenosize=\small}
%\SetLine %% new algorithm2e: \SetAlgoLined
\caption{Computing a banded approximation to $(t_iA+\xi_jI)^{-1}$.}\label{sparseInv}
\SetKwInOut{Input}{input}\SetKwInOut{Output}{output}

%%%%%%%%%%% INPUT %%%%%%%%%%%
\Input{$A\in\mathbb{R}^{n\times n},$ $A$ SPD, 
$t_i\in\mathbb{R}$, $\xi_j\in\mathbb{C}$ }
%$\nu\in\mathbb{N}$, $\epsilon_{S(t_i,\xi_j)},\epsilon_{spy},\tau>0$}
%%%%%%%%%%% OUTPUT %%%%%%%%%%%
\Output{$\reallywidehat{(t_iA+\xi_jI)^{-1}}$, $\reallywidehat{(t_iA+\xi_jI)^{-1}}\approx (t_iA+\xi_jI)^{-1}$}
%%%%%%%%%%%%%%%%%%%%%%%%%%%%%%%%%%%
\BlankLine
\nl Compute $t_iA-\xi_jI=L(t_i,\xi_j)D(t_i,\xi_j)L(t_i,\xi_j)^T$ \\
\For{$q = 1,\dots,n$ }{
%\For{$j=1,\ldots,\nu$}{
\nl Compute $\bar p_q\,(t_i,\xi_j)$ as in \eqref{bar_p} \\
\nl Set $\widehat p_q\,(t_i,\xi_j):= \min\{n,q+\bar p_q\,(t_i,\xi_j)\}$ \\
\nl $(y_q)_1=1/\left(L(t_i,\xi_j)\right)_{q,q}$ \\ 
\For{$k=q+1, \ldots, \widehat p_q\,(t_i,\xi_j)$}{
\nl $(y_q)_{k-q+1}=-\left(L(t_i,\xi_j)\right)^T_{k,q:k}(y_q)_{1:k-q}/\left(L(t_i,\xi_j)\right)_{k,k}$
}
\For{$k=q, \ldots, \widehat p_q\,(t_i,\xi_j)$}{
\nl $(z_q)_{k-q+1}=(y_q)_{k-q+1}/\left(D(t_i,\xi_j)\right)_{k,k}$
}
\nl $(s_q)_{\widehat p_q\,(t_i,\xi_j)-q+1}=(z_q)_{\widehat p_q\,(t_i,\xi_j)-q+1}/
\left(L(t_i,\xi_j)^T\right)_{\widehat p_q\,(t_i,\xi_j),\widehat p_q\,(t_i,\xi_j)}$\\
\For{$k=\widehat p_q\,(t_i,\xi_j)-1, \ldots, q$}{%
\nl \!\!\!{\small $(s_q)_{k-q+1}=\left((z_q)_{j-q+1}-\left(L(t_i,\xi_j)^T\right)^T_{k,k:\widehat p_q\,(t_i,\xi_j)}
(s_q)_{j-q+2:\widehat p_q\,(t_i,\xi_j)-q+1}\right)/
\left(L(t_i,\xi_j)^T\right)_{k,k}$ }\\
}
}
\nl Set $\mathfrak{S}=[s_1,\ldots,s_n]$ and $\mathfrak{s}:=\mbox{diag}(\mathfrak{S})$ \\
\nl Set $\reallywidehat{(t_iA+\xi_jI)^{-1}}:=\mathfrak{S}+\mathfrak{S}^T-\mbox{diag}(\mathfrak{s})$\\

\end{algorithm}
%%%

\section{Acknowledgment}
 We thank A. Haber for having provided us with the codes from \cite{Haber2016}.
 We thank the reviewers for their careful reading and several insightful remarks.

 Both authors are members of the Italian INdAM Research group GNCS.
\bibliography{sparseRHS}

\end{document}